\documentclass{article}
\usepackage{amsmath,amssymb,amsthm}
\usepackage{a4wide}
\usepackage{color}
\usepackage{enumerate}
\usepackage{cleveref}
\usepackage{backref}

\newtheorem{thm}{Theorem}[section]
\newtheorem*{thm*}{Theorem}
\newtheorem*{conj*}{Conjecture}

\newtheorem{lem}[thm]{Lemma}
\newtheorem{theorem}[thm]{Theorem}
\newtheorem{hypothesis}[thm]{Hypothesis}
\newtheorem{lemma}[thm]{Lemma}
\newtheorem{conjecture}[thm]{Conjecture}

\theoremstyle{remark}

\newtheorem{remark}[thm]{Remark}

\theoremstyle{definition}
\newtheorem{example}[thm]{Example}

\newcounter{claim}[thm]

\DeclareMathOperator{\Alt}{Alt}
\DeclareMathOperator{\aut}{Aut}
\DeclareMathOperator{\Sym}{Sym}
\DeclareMathOperator{\Aut}{Aut}

\DeclareMathOperator{\Soc}{Soc}

\newcommand{\LL}{\mathcal{L}}
\newcommand{\PP}{\mathcal{P}}
\newcommand{\Inn}{\mathrm{Inn}}
\newcommand{\Core}{\mathrm{Core}}

\newcommand{\soc}{\mathrm{Soc}}
\newcommand{\AGL}{\mathrm{AGL}}

\newcommand{\GL}{\mathrm{GL}}

\newcommand{\Sp}{\mathrm{Sp}}

\newcommand{\wh}{\widehat}
\newcommand{\R}{\mathcal{R}}

\title{Prime power coverings of groups}
\author{Michael Giudici, Luke Morgan and  Cheryl E. Praeger %\\
}
\begin{document}

\maketitle

\begin{abstract}

%Alternative abstract:

%\medskip
%{\color{blue}
For a finite group $A$ with normal subgroup $G$, a subgroup $U$ of $G$ is an $A$-prime-power-covering subgroup if $U$  meets every $A$-conjugacy-class of elements of $G$ of prime power order. It is conjectured that $|G:U|$ is bounded by some function of $|A:G|$, and this conjecture has number theoretic implications for relative Brauer groups of algebraic number fields. We prove the conjecture in the case that the action of $G$ on  the set of right cosets of $U$ in $G$ is innately transitive. This includes the case where $U$ is a maximal subgroup of $G$. The proof uses a new bound on the order of a  nonabelian finite simple group  in terms of its number of classes of elements of prime power order, which in turn depends on the Classification of the Finite Simple Groups.
%}
%
%\medskip 
%If $A$ is a finite group with normal subgroup $G$ of index $n$, and if $U$ is a subgroup of $G$ such that $U$ meets every $A$-conjugacy-class of elements of $G$ of prime power order, then it is conjectured that $|G:U|$ is bounded by some function of $n$. The conjecture has number theoretic implications for the relative Brauer groups of algebraic number fields.  A proof is given of the conjecture in the case where $U$ is maximal in $G$. The proof uses a new bound on the orders of finite simple groups in terms of the number of classes of elements of prime power order, which in turn depends on the Classification of the Finite Simple Groups.    

\medskip

\noindent{\bf Keywords:} covering subgroups, prime-power-covering, primitive groups\\
\noindent{\bf 2010 Mathematics Subject Classification:} 20B05, 20B15
\end{abstract}

%\section{Introduction: prime-power-covering sets}\label{ss:pp}

\section{Introduction}\label{s:intro}

%\subsection{Prime-power-covering sets and two conjectures}\label{ss:pp}

A famous result of Jordan \cite{Jordan} asserts that every finite transitive permutation group on a set of size at least two contains a derangement (an element with no fixed points). This is equivalent to the statement that it is not possible to cover the elements of a finite group by the set of conjugates of a proper subgroup. To see 
%this,
the equivalence, given a group $G$ with a subgroup $U$ consider the action of $G$ on the set of right cosets of $U$. Elements of $G$ fix a point in this action precisely if they are conjugate to an element of $U$ and so the guaranteed derangement is not contained in a conjugate of $U$. 

Jordan's result has  interesting applications in a wide variety of mathematical areas, see for example Serre's article \cite{Serre}. It has also motivated many questions about covering groups by conjugates of subgroups. 
One direction is to investigate how many conjugacy classes of subgroups are needed to cover the elements of a group. Such a covering is called a \emph{normal covering} and the number of conjugacy classes involved is called the \emph{normal covering number}. Groups with normal covering number two exist, for example when $n$ and  $q$ are even, the symplectic group $\Sp_n(q)$ can be covered by conjugates of its subgroups the orthogonal groups $\mathrm{O}^+_n(q)$ and $\mathrm{O}^-_n(q)$. All finite simple groups with normal covering number two have recently been classified by Bubboloni, Spiga and Wiegel \cite{BSW} who also gave an
%good
account of the many other areas of mathematics related to groups having normal covering number two.

Another natural direction to explore is to allow images of the subgroup under automorphisms of the group. Given $U<G \trianglelefteq A$, we say that $U$ is an \emph{$A$-covering subgroup} of $G$ if every element of $G$ is $A$-conjugate to an element of $U$. Such coverings were explored by Brandl \cite{Br} and are connected to questions about Kronecker classes of fields, see \cite{CP88,kron} for a  discussion of the links. 
%Saxl \cite[Proposition 2]{Saxl} proved that no proper covering subgroup exists for nonabelian finite simple groups and,
Saxl \cite[Proposition 2]{Saxl} proved that a nonabelian finite simple group $T$ has no proper $\Aut(T)$-covering subgroup and,
building on work of Jehne \cite[Theorem 5 and Remark (6.4)]{J77}, showed that if $M$ is Kronecker equivalent to a quadratic extension $K$ of an algebraic number field $k$ then $M=K$, answering a question posed in \cite{J77}. (For a finite extension $K$ of an algebraic number field $k$, the \emph{Kronecker set} of $K$ over $k$ is the set of all prime ideals of the ring of integers of $k$ having a prime divisor of relative degree one in $K$; and two finite extensions of $k$ are \emph{Kronecker equivalent} if their Kronecker sets have finite symmetric difference.) 

Neumann and the third author conjectured that the index in $G$ of an $A$-covering subgroup of $G$ is bounded by a function $f$ of $|A:G|$, see \cite[Problem 11.71]{Kourovka} and \cite[Conjecture 4.1]{kron}.   The Neumann-Praeger Conjecture is equivalent to the following: for an extension $L$ of a Galois extension $K/k$ of degree $n$ of algebraic number fields, if $L$ and $K$ are Kronecker equivalent over $k$ then the degree of $L$ over $K$ is bounded above by a function of $n$. Saxl's result implies that $f(2)=1$. 
%{\color{magenta}Luke: ``Should it be $f(2)=1$ since Saxl shows that $G$ has no proper $A$-covering subgroup when $|A:G|=2$?''}
 The third author   proved this conjecture in the cases where $U$ is a maximal subgroup of $G$ \cite[Theorem 4.4]{kron}.
Fusari, Harper and Spiga \cite{FHS} have proved the conjecture where the $A$-chief length of $G$ is bounded, correcting a subtle mistake in the first attempt of a proof of this result \cite[Theorem 4.6]{kron}. 
%given an alternative proof to the latter result that overcomes a subtle mistake in the original proof.
A surprising  reformulation  \cite[Question 6.1]{MRS} of the Neumann--Praeger Conjecture 
relates the conjecture to the problem of  bounding the degree of a permutation group whose derangement graph has no $k$-clique. This link is discussed in detail by Fusari, Previtali and Spiga (see \cite[Section 1.1]{FPS}) and they establish the veracity of 
%the footnote in \cite[Section 1.1]{FPS}), \cite[Theorem 1.1]{FPS} verifies  
\cite[Conjecture 4.1]{kron}
 in the case where the action of $G$ on the set of right cosets of $U$ is innately transitive, that is, has a transitive minimal normal subgroup \cite[Theorem 1.1]{FPS}.

Fein, Kantor and Schacher \cite[Theorem 1]{FKS} proved a dramatic extension of Jordan's result by showing that every finite transitive permutation group of degree at least two contains a derangement of prime power order. Whereas Jordan's result is a simple application of the Orbit-Counting Lemma, the only known proof of \cite[Theorem 1]{FKS} relies on the full strength of the Classification of the Finite Simple Groups. The motivation for this generalisation is that it is equivalent to the fact that the relative Brauer group of a nontrivial finite extension of global fields is infinite, see \cite[\S3]{FKS}.

The result of Fein, Kantor and Schacher can also be rephrased to say that the set of elements of prime power order in a group cannot be covered by the set of conjugates of a proper subgroup. This rephrasing naturally lends itself to exploring the notion of an $A$-prime-power-covering subset as follows.  For a subgroup $U$ of a finite group $A$, and an element $x\in A$ let $x^A=\{x^a : a\in A\}$ denote the $A$-class of $x$, and let $\PP(U) = \{ x\in U : |x| \ \text{is a prime power} \}$ denote the subset of elements of $U$ of prime power order. We define 
\begin{center}\label{def:ppc}
    $P_A(U):= \{ x^a : x\in\PP(U),\ a\in A\}$,\quad   the union of $A$-classes of elements of prime power order that intersect $U$ nontrivially,
\end{center}
and we call $P_A(U)$  the \emph{$A$-prime-power-covering set of $U$}.
\noindent
If subgroups $U$ and $G$ of $A$ are conjugate then they have the same $A$-prime-power-covering sets. It is also possible for subgroups that are not $A$-conjugate to have this property. Perhaps the smallest example is $A=\Alt(4)$ with $G=C_2^2$ and $U=C_2$, and the following small examples from \cite{CP88} show that $A$ can be insoluble.

\begin{example} %\cite[Examples 1--3]{CP88} 
In each of the following cases, $P_A(U)=P_A(G)$. 
\begin{enumerate}[(a)]
\item \cite[Examples 1--2]{CP88}\quad Take $A=\AGL_3(2)$ and $G=C_2^3.D_8$ a Sylow $2$-subgroup of $A$, and take $U= C_2\times D_8<G$, where the $C_2$ is the centraliser of $D_8$ in $C_2^3$.  For a second example take the same $A$ and $G$, and take $U=C_2\times C_4<G$.  
\item \cite[Example 3]{CP88}\quad Take $A=\GL_3(2)$ and $G\cong D_8$, a Sylow $2$-subgroup, and $U\cong C_4<G$. 
\end{enumerate}
\end{example}

We give several infinite families of examples of such triples $(A, G, U)$ in Section~\ref{s:examples}, but the small examples given here are sufficient to demonstrate that subgroups $U$ and $G$ of a group $A$ with equal $A$-prime-power-covering sets need not have the same order. 
Of course, the result of Fein, Kantor and Schacher implies that in the extreme case where $G=A$ there does not exist a proper subgroup $U<A$ with $P_A(U)=P_A(G)$. Thus questions about how different $U$ and $G$ can be, if $U$ and $G$ have equal $A$-prime-power-covering sets, concern proper subgroups $U$ and $G$ of $A$. One might conjecture that their orders cannot differ too much.

\begin{conjecture}\label{conj1}
   There is an increasing integer function $g$ such that, if $A$ is a finite group with subgroups $U$ and $G$ such that $P_A(U)=P_A(G)$ and $|A:G|=n$,  then $|A : U| < g(n)$. 
\end{conjecture}

In the case where $G\trianglelefteq A$, the property $P_A(U)=P_A(G)$ forces the  subgroup $U$ to be contained in $G$. 
%(and $A$ to induce a group of automorphisms of $G$ containing $\Inn(G)$). 
This particular case of Conjecture~\ref{conj1} is very similar to \cite[Conjecture $4.3'$]{kron} dating from 1994.

\begin{conjecture}\label{conj2}
   There is an increasing integer function $f$ such that, if $A$ is a finite group with normal subgroup $G$ of index $|A:G|=n$, and if $U\leqslant G$ is such that $P_A(U)=P_A(G)$, then $|G : U| < f(n)$. 
\end{conjecture}

Note that the conjectures coincide when $n=1$, that is, if $A=G$. In this case the truth of both conjectures follows from the Fein--Kantor--Schacher result which, as we mentioned, depends on the Classification of the Finite Simple Groups. This  suggests that the truth (or otherwise) of these conjectures may rely on rather deep group theory. 

If $U<G \trianglelefteq A$, then we say that $U$ is  an \emph{$A$-prime-power-covering subgroup of $G$} if $P_A(U)=P_A(G)$, that is, if $U$ meets every $A$-class of elements of $G$ of prime power order.  Conjecture~\ref{conj2} therefore asserts that the index in $G$ of an $A$-prime-power-covering subgroup of $G$ is bounded by a function of $|A:G|$. Guralnick \cite[Theorem 2.1]{G90} extended the results of Saxl and Fein--Kantor--Schacher to show that a finite nonabelian simple group $T$ does not have a proper $\Aut(T)$-prime-power-covering subgroup (see Theorem \ref{t:guralsaxl}, which we call the Guralnick--Saxl Theorem) and this again has consequences for relative Brauer groups, see \cite[Corollary B]{G90}. The Guralnick-Saxl Theorem confirms Conjecture~\ref{conj2} when $A$ is an almost simple group with socle $G$.  Further, an important special case of Conjecture~\ref{conj1} follows from \cite[Theorem A]{G90}, namely the case where the permutation group induced by $A$ on the set of $n$ cosets of $G$ is $\Sym(n)$ or $\Alt(n)$, and the proof of this result also relies on the Guralnick--Saxl Theorem.

% {\color{red}If $U<G \trianglelefteq A$ and $G=\{ x^a : x\in U, a\in A\}$, the subgroup 
% $U$ is called an \emph{$A$-covering subgroup of $G$}. Using} similar terminology, we say that $U$ is   an \emph{$A$-prime-power-covering subgroup of $G$} if $U$ meets every $A$-class of elements of $G$ of prime power order, that is, if $P_A(U)=P_A(G)$. 
% Conjecture~\ref{conj2} therefore asserts that the index in $G$ of an $A$-prime-power-covering subgroup of $G$ is bounded by a function of $|A:G|$. It is a strengthening of a conjecture of Neumann and the third author that makes the same assertion for $A$-covering subgroups.
% (The Neumann--Praeger conjecture appeared in 1990 in the $11^{th}$ issue of the Kourovka Notebook \cite[Problem 11.71]{Kourovka}, see also \cite[Conjecture 4.1]{kron}.) 
% All of these conjectures remain open and have number theoretic implications: {\color{red}the Neumann--Praeger Conjecture asserts that, for an extension $L$ of a Galois extension $K/k$ of degree $n$ of algebraic number fields, if $L$ and $K$ are Kronecker equivalent over $k$ then the degree of $L$ over $K$ is bounded in terms of $n$. The `prime power' Conjectures~\ref{conj1} and~\ref{conj2} have  consequences for the relative Brauer groups of algebraic number fields.  This is discussed further in \cite{G90, J77, kron}.}

The first aim of this paper is to show that the two conjectures above are equivalent.

\begin{theorem}\label{t:equiv}
    Conjecture~$\ref{conj1}$ is true if and only if Conjecture~$\ref{conj2}$ is true.
\end{theorem}

We prove Theorem~\ref{t:equiv} in Section~\ref{s:equiv}. 

%We note that the conjectures coincide if $n=1$, that is, if $A=G$. In this case the truth of the conjectures follows from the result of Fein, Kantor and Schacher \cite[Theorem 1]{FKS} Other  related results confirm this, for example, an important theorem of Guralnick and Saxl~\cite[Theorem 2.1]{G90} (see Theorem \ref{t:guralsaxl})  confirms Conjecture~\ref{conj2} when $A$ is an almost simple group with socle $G$. It shows that no simple group has a proper prime-power-covering subgroup, and its proof also relies on the Classification of the Finite Simple Groups. Further, an important special case of Conjecture~\ref{conj1} follows from \cite[Theorem A]{G90}, namely the case where the permutation group induced by $A$ on the set of $n$ cosets of $G$ is $\Sym(n)$ or $\Alt(n)$, and the proof of this result relies on the Guralnick--Saxl Theorem.

Turning to Conjecture~\ref{conj2}, if $(A,G,U)$ is a triple satisfying the hypothesis of that conjecture, it is natural to examine the properties of the permutation group induced by the action of $G$ on the set of right cosets of $U$ in $G$. If $U$ is a maximal subgroup of $G$, then this permutation group is \emph{primitive}. In this paper, we settle positively this case of the conjecture:

\begin{theorem}\label{t:main}
   There is an increasing integer function $f$ such that, if $A$ is a finite group with normal subgroup $G$ of index $n$, and if $U$ is a maximal subgroup of $G$ with $P_A(U)=P_A(G)$, then $|G : U| < f(n)$. 
\end{theorem}

 %CP: Add here some discussion of how this might be extended and what are the next steps. ?? Refer to a general analysis of the property $P_A(G)=P_A(U)$. We already saw in Lemma~\ref{l:Sunf} that the property is sometimes inherited by certain subgroups of $G, U$.
Theorem~\ref{t:main} is obtained as a particular case of a more general result; Theorem~\ref{t:mainIT}. In the hypothesis of Theorem~\ref{t:mainIT}, we only require that the action of $G$ on the set of right cosets of $U$ in $G$ is \emph{innately transitive}. In fact, we find that if this action of $G$   is innately transitive and  $U$ is an $A$-prime-power-covering subgroup of $G$, then $U$ \emph{must} be a maximal subgroup   (see Remark~\ref{rem:mustbeprim}).
Our approach to 
%the proof of Theorem~\ref{t:main}
Conjecture~\ref{conj2}
recasts the problem as one concerning permutation actions of $A$ and $G$ (see Lemma~\ref{l:cortriv}) and Theorems~\ref{t:main} and \ref{t:mainIT} treat the cases where these actions are primitive and innately transitive, respectively. The  general version of Conjecture~\ref{conj2} is equivalent to a problem about transitive permutation groups. We can show that the prime-power-covering property is inherited by certain sections of $G$, which play the role of the socle in a corresponding quotient action of the group $A$. Theorem~\ref{t:mainIT} and the examples in Section~\ref{s:examples} show that structure of the triples $(A,G,U)$ with $U$ an  $A$-prime-power-covering subgroup of $G$ is quite constrained. This structure will be explored further by the authors in forthcoming work.
%We note that if $U$ is not maximal, then the conclusion that $\hat{U}$ is maximal in Lemma~\ref{l:cortriv} must be removed, but the conclusion that $|G:U|=|\hat{G}:\hat{U}| $ still holds. This means that the problem can  be reinterpreted as a problem of transitive permutation groups, but is no longer a problem of primitive permutation groups.  Lemma~\ref{l:P_A(S) = P_A(UcapS)} shows that $P_A(X) = P_A(X \cap U)$ for arbitrary subgroups $X\leqslant G$ with $X$ normal in $A$, and we can also prove that $P_{A/Y}(X/Y) = P_{A/Y}( (X \cap U)Y/Y)$, where $X$ and $Y$ are $A$-invariant subgroups of $G$. Thus the property is inherited by both subgroups and by $A$-invariant sections of $G$. In forthcoming work, we will explore this problem for arbitrary subgroups $U$ further.

\begin{remark}\label{r:main}
   The proof  of Theorem~\ref{t:main} is given in Section~\ref{s:main}. It makes essential use both of the Guralnick--Saxl Theorem and also a new result \cite[Theorem 1.1]{GMP24}  on simple group orders, which we discuss in Section~\ref{ss:simple groups}.  We show that the function $f(n) = h(n+1)^{n}$ suffices, where $h$ is the increasing integer function in \cite[Theorem 1.1]{GMP24} (see Theorem~\ref{t:mT}).  We note that the third family of examples given in Section~\ref{s:examples} suggests how Theorem~\ref{t:mT} might be relevant for the proof of Theorem~\ref{t:main}.
\end{remark}

% \textcolor{blue}{Mention: 
% \begin{itemize}
%     \item Number theory motivation
%     \item EKR and derangements
% \end{itemize}}

\subsection{Acknowledgements} 
 This research was supported by the Australian Research Council Discovery Project grant DP230101268.  

\section{Preliminaries}\label{s:prelim}

\subsection{Simple groups}
\label{ss:simple groups}

For convenience, we state the Guralnick--Saxl Theorem \cite[Theorem 2.1]{G90} as it is critical to the proof of Theorem~\ref{t:GmustbeHSHC}.

\begin{theorem} [Guralnick--Saxl Theorem]\label{t:guralsaxl}
   Let $T$ be a finite nonabelian simple group, and let $T=\Inn(T)\leqslant A\leqslant \Aut(T)$. Then $T$ has no proper subgroup $U$ such that $P_A(U)=P_A(T)$. 
\end{theorem}

The new order bound for simple groups from \cite[Theorem 1.1]{GMP24} which we will apply in the proof of Theorem~\ref{t:main} involves the following group theoretic parameter.
For a finite group $G$, a subgroup $A\leqslant \Aut(G)$, and an element $g\in G$, the \emph{$A$-class}  of $g$ is the set $g^A=\{ g^\sigma\mid\sigma\in A\}$. If $A$ contains the inner automorphism group $\Inn(G)$ then $g^A$ is a union of $G$-conjugacy classes. Let 
\begin{equation}\label{mT}
m(G)=\max_{\mbox{primes}\ p} \#\{ \mbox{$\Aut(G)$-classes of elements of $p$-power order in $G$}\}.
\end{equation}

\begin{thm}\label{t:mT}(\cite[Theorem $1.1$]{GMP24})
There exists an increasing function $h$ on the natural numbers such that, for a finite nonabelian simple group $T$, the order of $T$ is at most $h(m(T))$.
\end{thm}

%\textcolor{blue}{**We use $f$ in the paper**  - but $h$ is not the eventual function $f$ we end up with, so we should not use $f$ in Theorem~\ref{t:mT}.}

\subsection{Finite innately transitive permutation groups}
\label{sec:IT gps}

%{\color{red}putting some preliminaries on innately transitive groups here}

A finite subgroup $G\leqslant \Sym(\Omega)$ is \emph{innately transitive} if there exists a transitive minimal normal subgroup. In particular, each finite primitive group is innately transitive. This class of permutation groups was first introduced in \cite{BP} and a complete structure theory was developed, similar in spirit to the O'Nan-Scott Theorem for primitive groups. A transitive minimal normal subgroup of $G$ is called a \emph{plinth}, and an innately transitive group   can have at most two plinths (\cite[Lemma 5.1]{BP}). 
% *****
\begin{theorem}[{\cite[Corollary 3.13]{PS}}]
\label{t:case div of IT}
Let $G$ be a finite innately transitive group. If $M$ and $N$ are two distinct plinths of $G$,  then $M\cong N \cong T^k$ for some nonabelian simple group $T$ and positive integer $k$, both $M$ and $N$ are regular, and $G$ is primitive. If $G$ has a unique plinth that is abelian, then $G$ is primitive of affine type.%An innately transitive group has at most two plinths. If there are two pln  
\end{theorem}
% / put in theorem environment
% If there exist two plinths, then $G$ is a primitive group of so-called HS (holomorph simple) or HC (holomorph compound) type (\cite[]{BP}). In the case of  two plinths, each plinth is the centraliser of the other, and both are regular (see \cite[Corollary 3.12]{PS}). If $G$ has a unique plinth, either the plinth is abelian and $G$ is a primitive group of HA (holomorph abelian) type, or, the plinth is nonabelian. In this last case there are many types (see \cite[Section 3]{BP}). 
% *****
We require one further result on such groups, which we give below. Recall that if $K$ and $L$ are normal subgroups of a group $G$ with $L < K$, then  $K/L$ is called a a \emph{chief factor} of $G$ if $K/L$ is a minimal normal subgroup of $G/L$.
If $\pi : G \rightarrow H$ is a surjective group homomorphism, and $K/L$ is a chief factor of $G$, then $K^\pi/L^\pi$ is either a chief factor of $H$, or trivial. If $K^\pi/L^\pi \neq 1$, then $K/L \cong K^\pi/L^\pi$ and the actions of $G$ on $K/L$ and of $H$ on $K^\pi/L^\pi$ are equivalent (see \cite[Chapter 10]{Isaacs} for further details). 

\begin{lemma}\label{l:ITgps}
    Let $G \leqslant \Sym(\Omega)$ be an innately transitive group with plinth $M \cong T^s$. If $K/L$ is a chief factor of $G$ such that $K/L \cong T^s$, then $K$ contains a plinth of $G$. In particular, if $M$ is the unique plinth of $G$, then $M \leqslant K$.
\end{lemma}
\begin{proof}
Let $K$ and $L$ be normal in $G$ and  such that $K/L \cong M = T^s$ is a chief factor of $G$. Since $M$ is a minimal normal subgroup of $G$, either $K \cap M = 1$ or $K \cap M = M$. The latter implies that $M \leqslant K$, and all assertions follow in this case. We may assume therefore that $K \cap M = 1$ and so $K$ centralises $M$, that is, $K \leqslant C_G(M)$. Since $|C_G(M)| \leqslant |M|$ (see \cite[Theorem 3.2(i)]{PS}) % Citation from Bamberg Praeger about C_G(plinth) being section of plinth
we have
$$|M| = |K/L| \leqslant |K| \leqslant |C_G(M)| \leqslant |M|.$$
In particular $L=1$,  so  $K$ is a minimal normal subgroup of $G$, and also $K=C_G(M)$.  Hence, \cite[Theorem 3.2]{PS} shows that $K$  is regular, and thus $K$ itself is a plinth. In particular, $K$ and $M$ are \emph{distinct} plinths.
%It follows that in the case where $G$ has a unique plinth, $M \leqslant K$ must hold.
\end{proof}

We will prove a  much stronger  result than Theorem~\ref{t:main} from which Theorem~\ref{t:main} follows as an immediate corollary -- for the same function $f$.

\begin{theorem}\label{t:mainIT}
   There is an increasing integer function $f$ such that, if $A$ is a finite group with normal subgroup $G$ of index $n$, and if $U$ is a proper subgroup of $G$ with $P_A(U)=P_A(G)$ and the $G$-action on $[G:U]:= \{ Ug \mid g \in G \}$ is innately transitive, then $|G : U| < f(n)$. 
\end{theorem}

\begin{remark}
    \label{rem:mustbeprim}
    One ingredient in the proof of Theorem~\ref{t:mainIT} is Theorem~\ref{t:GmustbeHSHC}, proved in Section~\ref{sec:unique plinth}. The latter theorem essentially says that the innately transitive action of $G$ on the set of right cosets of $U$ either has two plinths (and so is primitive by Theorem~\ref{t:case div of IT}), or has a unique plinth which is abelian and is thus a primitive group of affine type. In particular,  this says that if the hypothesis of Theorem~\ref{t:mainIT} is satisfied, then $U$ is a maximal subgroup of $G$.
\end{remark}

\subsection{Vertex  colourings of graphs}

In the proof of Theorem~\ref{t:main} we use some tools from graph theory. For a graph $\Gamma = (V, E)$ with vertex set $V$ and edge set $E$, we define the \emph{degree} of a vertex to be the number of edges incident with it. The \emph{maximum degree} of $\Gamma$ is then the maximum of the degrees of all vertices of $\Gamma$.
A \emph{proper vertex colouring} of a graph $\Gamma=(V,E)$ by a set $\mathcal X$ of colours is a function $f  :V \rightarrow \mathcal X$ such that for any edge $\{u,v\} \in E$ we have $f(u) \neq f(v)$. If $|\mathcal X| = k$, we say that $f$ is a \emph{proper vertex $k$-colouring}; the \emph{chromatic number} of $\Gamma$ is the least integer $k$ such that there exists a  proper vertex $k$-colouring. The chromatic number of a  graph may be bounded in terms of various structural properties.  We  will employ the following upper bound, the proof of which is a standard illustration of the `greedy algorithm'. For a statement see \cite[5.1.2, F13, pg.~345]{HoGT}.

\begin{thm}
    \label{thm:colouring}
   Each finite  graph $\Gamma$ with maximum degree $d$ admits a proper vertex $(d+1)$-colouring.
\end{thm}

%{\color{red}CP: unsure how to organise this. Feels uncomfortable to state these results in the intro as they are not proved in the paper, but we want to refer to the statements for the proof. We could move to here the discussion about the greedy algorithm for vertex-colouring of graphs - used in the proof}

\section{Equivalence of the two conjectures}\label{s:equiv}

As discussed in Section~\ref{s:intro}, Conjecture~$\ref{conj2}$ is a special case of Conjecture~$\ref{conj1}$ and both conjectures are true in the case $n=1$. Thus to prove Theorem~\ref{t:equiv} we need to prove that Conjecture~$\ref{conj2}$ implies Conjecture~$\ref{conj1}$ for each $n\geqslant 2$.

\medskip\noindent
\emph{Proof of Theorem~\ref{t:equiv}.}\quad 
Suppose that Conjecture~$\ref{conj2}$ is true for an increasing integer function $f$, and let $A$ be a finite group with subgroups $U$ and $G$ such that $P_A(U)=P_A(G)$, and $|A:G|=n\geqslant 2$. Our task is to bound $|A:U|$ by an increasing function of $n$, so we may in particular assume that  $U$ is a proper subgroup of $A$.  

Let $N=\Core_A(G):=\cap_{a\in A} G^a$. Then the $A$-action by right multiplication on $[A:G]=\{ Ga\mid a\in A\}$  induces a homomorphism $\phi:A\to \Sym([A:G])$ with kernel $N$, and hence $n':=|A:N|\leqslant n!$. If $N=1$ then we have $|A:U|\leqslant |A|=n'\leqslant n!$. Thus we may assume that $N\ne 1$.

We claim that $P_A(N)= P_A(U\cap N)$. Since $N\unlhd A$, we have $x^A\subseteq N$ for each $x\in U\cap N$ of prime power order, and hence $P_A(U\cap N)\subseteq P_A(N)$. Conversely if $x\in P_A(N)$, then $x=y^a$ for some $y\in N$ of prime power order and some $a\in A$, and hence $x\in y^A\subseteq N\subseteq G$. Since $P_A(G)=P_A(U)$ it follows that $x = (y')^{a'}$ for some $y'\in U$ of prime power order, and some $a'\in A$. Thus $y'\in x^A\subseteq N$, and so $y'\in U\cap N$ and hence $x=(y')^{a'}\in P_A(U\cap N)$. Therefore $P_A(N)\subseteq P_A(U\cap N)$ and so equality holds, proving the claim.

Thus the hypotheses of Conjecture~\ref{conj2} hold for the group $A$ with normal subgroup $N$ of index $n'$ and subgroup $U\cap N$ of $N$; and since we are assuming that Conjecture~\ref{conj2} holds relative to the increasing integer function $f$, we have $|N:U\cap N|\leqslant f(n')\leqslant f(n!)$. Thus
\[
|A:U| = |A:UN|\cdot |UN:U| = |A:UN|\cdot |N:U\cap N|\leqslant |A:N|\cdot |N:U\cap N|\leqslant n!\cdot f(n!).
\]
Thus we have shown that, if Conjecture~\ref{conj2} holds relative to an increasing integer function $f$, then Conjecture~\ref{conj1} holds with the increasing integer function $n\to n! f(n!)$. \hfill \qed

\section{Examples}
\label{s:examples}

In light of Theorem~\ref{t:equiv}, it is sufficient to give examples of groups satisfying the conditions of Conjecture~\ref{conj2}. Thus we are interested in a group $A$ with a normal subgroup $G$ and a proper subgroup $U < G$ such that $P_A(G) = P_A(U)$. The normality of $G$ in $A$ means that $P_A(G) = \PP(G)$. For such triples $(A,G,U)$ we are tasked with understanding the relationship between $|A:G|$ and $|G:U|$.

\begin{enumerate}
\item  \emph{Examples with $G$ elementary abelian:}\quad Let $p$ be a prime and let $d$ be a positive integer. Let $A= \AGL(d,p) = G \rtimes \GL(d,p)$, $G \cong C_p^d$ and, regarding $G$ as a $d$-dimensional $\mathbb{F}_p$-space, let $U$ be any non-zero subspace of $G$. Since $\GL(d,p)$ is transitive on the non-zero vectors of $G$, we have  $P_A(U)=P_A(G)=G$. Note that $A$ is naturally an affine permutation group of degree $p^d$, and we can obtain further examples with the same $U$ by taking $A=G \rtimes H$ where $H$ is any subgroup of $\GL(d,p)$ transitive on non-zero vectors. For these groups we see that $|A:G| = |H| \geqslant p^d-1$ and  $|G:U|  \leqslant \frac{p^d}{p} = p^{d-1}$. Thus $|G:U| < |A:G|$. For these examples, the closest these quantities can get occurs when $H=\GL_1(p^d)$ and $U=C_p$, where we have $|A:G|/p<|G:U|<|A:G|$. 

\item \emph{Examples with $G$ a non-abelian $r$-group:}\quad Let $r$ be an odd prime and  let $m$ be a positive integer. Let $p$ be a prime such that $r$ divides $p-1$. In the linear groups $\GL(r^m,p)$, the `extraspecial normaliser' $A = r_+^{1+2m} \rtimes \Sp(2m,r)$ has a normal subgroup $G=r_+^{1+2m}$, which is a non-abelian group of order $r^{1+2m}$ and exponent $r$. Let $U$ be any  subgroup of $G$ properly containing $Z(G)$. Since $A/Z(G) \cong r^m \rtimes \Sp(2m,r)$ is an affine $2$-transitive group (such as in Example 1 above), we see that every element of $G/Z(G)$ is $A/Z(G)$-conjugate to an element of $U/Z(G)$. Thus for any $g\in G$, there is an element $a\in A$ such that $g^a \in UZ(G) = U$, and it follows that  $P_A(G)=P_A(U)$. We may  find examples where $G$ is a non-abelian $2$-group in a similar way. Note that since $U$ can be any  subgroup of $G$ properly containing $Z(G)$, the subgroup $U$ may be abelian or, (if $m\geq2$) nonabelian. 

\item \emph{Examples with $G$ a non-abelian characteristically simple group:}\quad  Let $T$ be a non-abelian simple group and let $k \geqslant m_0(T)+1$, where $m_0(T)$ is the maximum over all primes $p$, of the number of  conjugacy classes of $p$-elements in $T$. (For convenience we work with $m_0(T)$, which is in general slightly larger than $m(T)$ given in \eqref{mT}.) Set $G=T^k$ and let  $A = T \wr \Sym(k) = G \rtimes \Sym(k)$. Let $U = \{ (t_1,t_2,\ldots,t_{k-2},t,t) \mid t_i,t\in T\}$. Then $U$ is a maximal subgroup of $G$. If $x=(x_1,\ldots,x_k)\in G$ is a $p$-element, then since $k \geqslant m_0(T)+1$, there exist distinct $i, j$ such that $x_i$ and $x_j$ belong to the same conjugacy class in $T$. Let  $t=(t_1,\ldots,t_k) \in G$ be such that $x_j^{t_j} = x_i$ and $t_s=1$ for $s\neq j$, and let $\sigma \in \Sym(k)$ be such that $i^{\sigma}=k-1$ and $j^{\sigma} = k$. Then $t\sigma \in A$ and the $k^{th}$ and $(k-1)^{st}$ entries of $x^{t \sigma} $ are $x_i$ and $x_j^{t_j}=x_i$, respectively, so $x^{t\sigma} \in U$. Hence $P_A(G) = P_A(U)$. 

This example may be further generalised by replacing $\Sym(k)$ by any $2$-transitive group $H$ of degree $k$. In an extreme case, when $k$ is a prime power, $H$ can be taken to be a sharply $2$-transitive group, so that $|H|=k(k-1)$. In this instance, we may take $k$ as the smallest prime power greater than $m_0(T)$, and we note that $k\leq 2m_0(T)$ (by Bertrand's Postulate, see \cite[Theorem 8.7]{NZM}). Then $|A:G| = |H| = k(k-1)<4m_0(T)^2$ and $|G:U| = |T|$. Thus in this case Conjecture~\ref{conj2} essentially asserts that $|T|$ is bounded above in terms of $m_0(T)$ (suggesting how Theorem~\ref{t:mT} may be relevant to the proof of Conjecture~\ref{conj2}). The smallest possible example in this vein is exhibited with $T=\Alt(5)$, $k=4$ and $H=\Alt(4)$. Here $A=\Alt(5) \wr \Alt(4)$, $|A:G|=12$ and $|G:U|=|\Alt(5)|=60$.
\end{enumerate}

\section{Bound for innately transitive actions}\label{s:main}

In this section we prove Theorem~\ref{t:mainIT}. Throughout the section, the following hypothesis and notation will hold:

\begin{hypothesis}\label{hyp:main}
   Let $A$ be a finite group with normal subgroup $G$ of index $n$ and  $U$ a proper subgroup of $G$.  Set 
   $$\Omega = [A:U] = \{ Ua \mid a \in A \}.$$ 
   This means that the $A$-coset action by right multiplication on $\Omega$ is transitive
   % defines a   homomorphism $\phi:A\to \Sym(\Omega)$
   of degree 
   $$|\Omega|=|A:U|=n\cdot |G:U|$$ and,  for the `point' $\alpha = U$ of  $\Omega$,  the stabiliser $A_\alpha = G_\alpha = U$.   %and we have $U$ a maximal (proper) subgroup of $G$ with trivial $A$-core, that is, $\Core_A(U)=1$.
   Since %$A \phi \leqslant \Sym(\Omega)$ acts transitively, and as
   $G\unlhd A$ and $G$ contains the point stabiliser $U$, it follows that 
   $$\Omega = \Omega_1 \cup \dots \cup \Omega_n$$
   with  each $\Omega_i$ a $G$-orbit and $\alpha\in\Omega_1$.   Further, $A$ permutes the $\Omega_i$ transitively, hence the groups $G^{\Omega_i}$ are  permutation equivalent. Finally, for $1\leqslant i \leqslant n$, we set:
   $$K_i = G_{(\Omega_i)}, \quad \text{the kernel of the action of } G \text{ on } \Omega_i.$$
\end{hypothesis}

%Since $G$ is a normal subgroup of $A$, $G$ has $|A:G| = n$ orbits on $\Omega$, 

  We first record a simple observation that the  property $P_A(G)=P_A(U)$ is inherited by certain subgroups of  $G$.

\begin{lem}
\label{l:P_A(S) = P_A(UcapS)}
Suppose that $X \leqslant G $ and $X$ is normal in $A$. If $P_A(G)=P_A(U)$, then we have the equality $P_A(X) = P_A(U \cap X)$.
\end{lem}
\begin{proof}
        Since $X$ is normal in $A$, $P_A(U\cap X)\subseteq P_A(X)$.  For the reverse inclusion let $x\in X$ have prime power order. Then $x\in P_A(G)=P_A(U)$ so $x=y^a$ for some $a\in A$, and $y\in U$. This implies that $y=x^{a^{-1}}\in U\cap X^{a^{-1}}=U\cap X$,  so $x=y^a\in P_A(U\cap X)$ . Thus equality holds.
\end{proof}

Now let 
\[
N=\Core_A(U) = \bigcap_{a\in A} U^a,
\]
the \emph{$A$-core of $U$} which is equal to the kernel of the action of $A$ on $\Omega$. We next observe that the property $P_A(G)= P_A(U)$ is inherited by the quotient group $A/N$ whilst preserving both the index $|G:U|$ and the permutation group induced on $[G:U]$ by $G$.
 
\begin{lem}\label{l:cortriv}
    %For each subgroup $H\leqslant A$, write $\widehat{H}:=HN/N$. Then $\widehat{A}$ is a finite group  with normal subgroup $\widehat{G}$ of index $n$,   $\widehat{U}$ is a subgroup of $\widehat{G}$ with index equal to $|G:U|$ and the action of $\widehat{G}$ on the set of cosets of $\widehat{U}$ in $\widehat{G}$ is permutation equivalent to the action of $G$ on the set of cosets of $U$ in $G$. Moreover, $P_A(G)=P_A(U)$ if and only if $P_{\widehat{A}}(\widehat{U})=P_{\widehat{A}}(\widehat{G})$.
    For each subgroup $H\leqslant A$, write $\widehat{H}:=HN/N$. Then $\widehat{A}$ is a finite group  with normal subgroup $\widehat{G}$ of index $n$,   $\widehat{U}$ is a subgroup of $\widehat{G}$ with  $|\wh{G}:\wh{U}|=|G:U|$ and the $\wh{A}$-action on $[\wh{A}:\widehat{U}]$ is permutation equivalent to the $A$-action on $[A:U]$. Moreover, $P_A(G)=P_A(U)$ if and only if $P_{\widehat{A}}(\widehat{U})=P_{\widehat{A}}(\widehat{G})$.
\end{lem}

\begin{proof}
  %  {\color{red}Luke: I think we could just quote this from somewhere: By the Correspondence Theorem $\widehat{G}$ is normal in  $\widehat{A}$ of index $n$, and as $U$ contains $N$,   $\widehat{U}$ is a subgroup of $\widehat{G}$ with index equal to $|G:U|$.   For $g\in G$ we define $\widehat{g}=Ng$, then the map $\rho : G/U \rightarrow \widehat{G}/\widehat{U}$ defined by $\rho : Ug \mapsto \widehat{U}\widehat{g}$ induces the permutation equivalence (it is part of the correspondence theorem quoted above that $\rho$ is a bijection).  For $h\in G$, we have 
  % $$((Ug)^h)\rho = (Ugh)\rho = \widehat{U}\widehat{gh} = \widehat{U}\widehat{g}\widehat{h}= (\widehat{U}\widehat{g})^{\widehat{h}}= ((Ug)\rho)^{\widehat{h}}$$
  % as required.} 
      By the Correspondence Theorem $\widehat{G}$ is normal in  $\widehat{A}$ of index $n$, and as $U$ contains $N$,   $\widehat{U}$ is a subgroup of $\widehat{G}$ with index equal to $|G:U|$.   For $a\in A$ we define $\widehat{a}=Na$, then the map $\rho : [A:U] \rightarrow [\widehat{A}:\widehat{U}]$ defined by $\rho : Ua \mapsto \widehat{U}\widehat{a}$ induces the permutation equivalence (it is part of the Correspondence Theorem quoted above that $\rho$ is a bijection).  For $b\in A$, we have 
  $$((Ua)^b)\rho = (Uab)\rho = \widehat{U}\widehat{ab} = \widehat{U}\widehat{a}\widehat{b}= (\widehat{U}\widehat{a})^{\widehat{b}}= ((Ua)\rho)^{\widehat{b}}$$
  as required.
    
    For the final assertion, suppose first that $P_A(G)=P_A(U)$. Since $\wh{U} \leqslant \wh{G}$,  $P_{\widehat{A}}(\widehat{U})\subseteq P_{\widehat{A}}(\widehat{G})$. To prove the reverse inclusion, let $xN\in \widehat{G}$ be of order $p^r$ for some prime $p$ and integer $r\geqslant1$. Then, replacing $x$ if necessary by a power $x^b$ with $b$ coprime to $p$,  we may assume that $|x|=p^s$ for some $s\geqslant r$. Since $P_A(U)=P_A(G)$, it follows that  $x=y^a$ for some $y\in U$ and $a\in A$. Thus $xN = (yN)^{aN}\in P_{\widehat{A}}(\widehat{U})$, proving that $P_{\wh{A}}(\wh{G})=P_{\wh{A}}(\wh{U})$.

    Suppose conversely that $P_{\wh{A}}(\wh{G})=P_{\wh{A}}(\wh{U})$ and let $x\in G$ have prime power order. Then $xN$ has prime power order in $\wh{A}$ and so there is a coset $aN \in \wh{A}$ such that $(xN)^{aN} = x^aN \in \wh{U}$. Thus $x^a \in NU = U$, and so $P_A(G) = P_A(U)$, as required.
\end{proof}

It follows from Lemma~\ref{l:cortriv} that, if Conjecture~\ref{conj2} holds, with some increasing integer function $f$, in the case where $U$  has trivial $A$-core and the $G$-action on $[G:U]$ is innately transitive, then it holds in general,  with the same function $f$, for arbitrary subgroups $U$ of $G$ such that the $G$-action on $[G:U]$ is innately transitive. 
Thus  from now on we will adopt the following hypothesis:
%that $N=1$, i.e. that $A$ acts faithfully on $\Omega$.

\begin{hypothesis}\label{h:main2}
    Hypothesis~\ref{hyp:main} holds and additionally, $U$ is core-free in $A$ and $G^{\Omega_i}$ is innately transitive, for each $i$.
\end{hypothesis}

  %We will assume that $G^{\Omega_i}$ is an innately transitive group. 
  
  We will divide our analysis into three cases: (i) $G^{\Omega_i}$ has an abelian plinth (the affine case); (ii) $G^{\Omega_i}$ has a unique plinth that is nonabelian; (iii) $G^{\Omega_i}$ has at least two plinths. These cases will be explored in the subsequent subsections.

\subsection{Affine type}

In this section we assume that Hypothesis~\ref{h:main2} holds and that $G^{\Omega_i}$ has  an abelian plinth. By Theorem~\ref{t:case div of IT}, $G^{\Omega_i}$ is primitive of affine type. 

% We consider first the case where this induced action is cyclic of prime degree, in other words, the case where $U$ is normal in $G$.

% \begin{lem}\label{l:Unorm}
%     If $U\unlhd G$, then $|G:U|\leqslant n$.
% \end{lem}

% \begin{proof}
%     Suppose that $U\unlhd G$. Then the primitive group $G^{\Omega_1}=G/U\cong C_p$ for a prime $p$. As the action on $\Omega$ is faithful, $G$ is isomorphic to a subgroup of $\prod_{i=1}^n G^{\Omega_i}=C_p^n$, so $G$ and hence also $U$ are $p$-groups. Thus $P_A(U)=P_A(G)=G$, and so $G=\cup_{a\in A} U^a$ is a union of at most $n$ distinct $A$-conjugates of $U$. Hence $|G|\leqslant n|U|$. 
% \end{proof}

% From now on we assume that $U$ is not normal in $G$, so

\begin{lem}\label{l:affine}
Suppose that $G^{\Omega_i}$ is a primitive group of affine type. If $P_A(G)=P_A(U)$, then $|G:U| \leqslant n$.
\end{lem}
\begin{proof}
Let $S = \soc(G)$ and let $X$ be a minimal normal subgroup of $G$, so $X\leqslant S$. Since $X \neq 1$, and $G$ acts faithfully on $\Omega$, there exists $i$ such that $X^{\Omega_i}\neq 1$, and so  $X^{\Omega_i}$ is a minimal normal subgroup of $G^{\Omega_i}$. Since $G^{\Omega_i}$ is an affine primitive group, the induced group $X^{\Omega_i} = \soc(G^{\Omega_i})$, which  is elementary abelian, and we also have $S^{\Omega_i} = \soc(G^{\Omega_i})$.  Since $X\leqslant S\leqslant \prod_{i=1}^n S^{\Omega_i}$, it follows from the minimality of $X$ that $X \cong X^{\Omega_i}=S^{\Omega_i}$. This implies that $X$ and $S$ are elementary abelian groups. In particular, every element of $S$ is a $p$-element for some prime $p$. Since $A$ acts transitively on $\{\Omega_1,\ldots,\Omega_n\}$, it follows that for any $i$ we have $S^{\Omega_i} = \soc(G^{\Omega_i})$ and so $S$ acts regularly on each $\Omega_i$. In particular $U\cap S=S_\alpha$ fixes $\Omega_1$ pointwise, so $U \cap S = S_{(\Omega_1)}$ and $|S:U\cap S|=|\Omega_1|=|G:U|$. For each $a\in A$, the image $\Omega_1^a=\Omega_i$ for some $i$, and we have $(U \cap S)^a=U^a\cap S = S_{(\Omega_i)}$. By Lemma~\ref{l:P_A(S) = P_A(UcapS)} the equality $P_A(G)=P_A(U)$ implies that $P_A(S)=P_A(U \cap S)$. Therefore, since every element of $S$ is a $p$-element, 
\[ 
S= P_A(S)= \bigcup_{a\in A} (U\cap S)^a = \bigcup _{i=1}^n S_{(\Omega_i)},
\]
% and there are at most $n$ distinct $A$-conjugates of the subgroup {\color{green}$U\cap S=S_{(\Omega_1)}$}.  
Hence $|S|\leqslant n\cdot |U\cap S|$, and so $|G:U|=|S:U\cap S|\leqslant n$. 
\end{proof}

% \begin{lem}\label{l:2plinths}
% Suppose that $G^{\Omega_i}$ has two plinths, each isomorphic to $T^k$ for some nonabelian finite simple group $T$ and integer $k$. Then $\soc(G) = S_1 \times \cdots \times S_t$, where each $S_i \cong T^k$.    
% \end{lem}
% \begin{proof}
%    Since the $G^{\Omega_i}$ are equivalent primitive permutation groups it follows that they have isomorphic socles $T^k$ for some finite simple group $T$ and some $k\geqslant 1$. Then as $G$ is isomorphic to a subgroup of $\prod_{i=1}^n G^{\Omega_i}$ it follows that $S=\Soc(G)\cong T^m$ where $k\leqslant m\leqslant nk$. 
% \end{proof}

% \begin{lem}\label{l:soc}
% Suppose that $G^{\Omega_i}$ has a unique nonabelian plinth. Then there exists a characteristically simple normal subgroup $S$ of $A$ that is contained in $G$. Moreover, let $T$ and $k$ be such that the plinth of $G^{\Omega_i}$ is isomorphic to $T^k$. Then $S = S_1 \times \cdots \times S_t$, where each $S_i$ is a minimal normal subgroup of $G$ and $S_i \cong T^k$.
% %    The group $G$ has socle $S\cong T^m$ for some finite  simple group $T$ and some $m\geqslant 1$, and for each $i$, $\Soc(G^{\Omega_i})=T^k$, where $k\leqslant m\leqslant nk$ and $k$ is independent of $i$. Moreover, if $T$ is cyclic, then  $|G:U|\leqslant n$.
% \end{lem}

\subsection{Unique nonabelian plinth}
\label{sec:unique plinth}

In this section we assume that Hypothesis~\ref{h:main2} holds and that $G^{\Omega_i}$ has  a unique plinth that is nonabelian. Since the groups $G^{\Omega_i}$ are equivalent for all $i$, this condition holds for all $i$. 
%Next we examine the structure of $G$ when $G^{\Omega_i}$ has a unique nonabelian plinth.
%In the division of innately transitive groups, the groups considered here correspond to the types: AS, ASQ, SD, CD, TW, PA, PQ, DQ. 

\begin{lem}\label{l:soc}
Suppose that for each $i$ the   group $G^{\Omega_i}$ has a unique plinth $P_i \cong  T^s$ for some finite nonabelian simple group $T$ and positive integer $s$. \begin{enumerate}
    \item[(a)]  There exists a unique minimal normal subgroup $S_i$ of $G$ such that $S_i \cong S_i^{\Omega_i} = P_i$.
\end{enumerate}
\noindent    Let $S:=\langle S_1,\ldots,S_n \rangle$.
\begin{enumerate}
    \item[(b)] There exists a positive divisor $t$ of $n$ such that, after relabelling the $S_i$ if necessary, we have $S =  S_1 \times \cdots \times S_t \cong T^{st}$ and $S^{\Omega_i} = S_i^{\Omega_i} \cong S_i$ for $1\leqslant i \leqslant t$.
\end{enumerate}  
% If the group $G^{\Omega_i}$ is not of affine type, then $G^{\Omega_i}$ has exactly two plinths. In particular, $G^{\Omega_i}$ is primitive and $\soc(G) = S_1 \times \cdots \times S_t$, where each $S_i$ is a minimal normal subgroup of $G$ and $S_i \cong T^s$ for some finite nonabelian simple group $T$ and positive integer $s$.
%    The group $G$ has socle $S\cong T^m$ for some finite  simple group $T$ and some $m\geqslant 1$, and for each $i$, $\Soc(G^{\Omega_i})=T^s$, where $k\leqslant m\leqslant nk$ and $k$ is independent of $i$. Moreover, if $T$ is cyclic, then  $|G:U|\leqslant n$.
\end{lem}
\begin{proof}
%Suppose for a contradiction that $G^{\Omega_i}$ has a unique nonabelian plinth.
For $1\leqslant i \leqslant n$  let $P_i  \leqslant G^{\Omega_i}$ denote the unique plinth of $G^{\Omega_i}$ and recall that $K_i = G_{(\Omega_i)}$,  the kernel of the action of $G$ on $\Omega_i$. Thus for any subgroup $H$ of $G$ we have $H^{\Omega_i} = HK_i/K_i$. Further, let $E_i$ be the full preimage in $G$ of $P_i$. This means that $K_i < E_i$, and if $H \leqslant G$ is such that $H^{\Omega_i} = P_i$, then $H \leqslant E_i$. 

Let $\mathcal X_i = \{ X \unlhd G \mid X^{\Omega_i} = P_i \}$, the set of normal subgroups of $G$ that project to $P_i$. In particular, $E_i \in \mathcal X_i$, so  $\mathcal X_i \neq \emptyset$.  

\begin{claim}\label{claim:closedunderintersection}
    $\mathcal X_i$ is closed under intersection.
\end{claim}

\medskip
Let $X, Y \in \mathcal X_i$. Since $X$ and $Y$ are normal in $G$, the intersection $X \cap Y$ is normal in $G$, and moreover $[X,Y] \leqslant X \cap Y$ with $[X,Y] \unlhd G$. Now, $[X,Y]^{\Omega_i} = [X^{\Omega_i}, Y^{\Omega_i}] = [P_i,P_i] =P_i$ (as $P_i=T^s$ is nonabelian). In particular, $P_i \leqslant (X\cap Y)^{\Omega_i} \leqslant X^{\Omega_i} = P_i$, and so $X\cap Y \in \mathcal X_i$. Thus $\mathcal X_i$ is indeed closed under intersection.

\medskip

For $1\leqslant i \leqslant n$ we set:
$$S_i := \bigcap _{X \in \mathcal X_i} X.$$
Notice that $S_i \unlhd G$, and by \ref{claim:closedunderintersection}, $S_i\in\mathcal{X}_i$. Thus $S_i$ is the unique minimal element of $\mathcal{X}_i$ (by inclusion), and if $X \unlhd G$ and $X^{\Omega_i} = P_i$, then $S_i \leqslant X$. Further, since $A$ acts transitively on $\{\Omega_1, \ldots, \Omega_n\}$ and normalises $G$, we see that $A$ permutes the set $\{\mathcal{X}_1,\ldots,\mathcal{X}_n\}$ transitively, and also the set of subgroups $\{S_1,\ldots,S_n\}$ transitively. In particular, $|S_i| = |S_j|$ for all $i,j$.

\begin{claim}\label{claim:x=si or x in si cap ki}
    If $X$ is a normal subgroup of $G$ contained in $S_i$, then either  $X= S_i$ or $X \leqslant S_i \cap K_i$.
\end{claim}

\medskip

Let $X$ be such a subgroup. Then $X^{\Omega_i}$ is normal in $G^{\Omega_i}$ and $X^{\Omega_i} \leqslant S_i^{\Omega_i}$. Since $S_i /(S_i \cap K_i) = S_i^{\Omega_i} = P_i$ is a minimal normal subgroup of $G^{\Omega_i}$, either $X^{\Omega_i}=S_i^{\Omega_i}=P_i$ or $X^{\Omega_i}=1$. In the former case,    $X^{\Omega_i} = P_i$ and hence $X \in \mathcal X_i$, which implies that $S_i \leqslant X$ and thus $X=S_i$ since $X \leqslant S_i$ by assumption. In the latter case,  $X \leqslant K_i$ and so $X\leqslant S_i \cap K_i$, as required.

\begin{claim}\label{clam:projs plinth}
    For all $i,j$, either  $S_i^{\Omega_j} = 1$, or, $S_i=S_j$ and $S_i^{\Omega_j} = P_j$.
\end{claim}

\medskip

Since $S_i \cap K_j$ is normal in $G$ and contained in $S_i$, it follows from \ref{claim:x=si or x in si cap ki} that either  (a) $S_i \cap K_j = S_i$, or, (b) $S_i \cap K_j \leqslant S_i \cap K_i$. In case (a), we see 
% $$S_i^{\Omega_i} = 
% %(S_i \cap K_i)^{\Omega_i}(S_i \cap K_j)^{\Omega_i} = 
% (S_i \cap K_j)^{\Omega_i}$$
% and from the minimality of $S_i$, this implies
$S_i=S_i \cap K_j \leqslant K_j$, so $S_i^{\Omega_j}=1$. Suppose now that case (b) holds. By the Third Isomorphism Theorem, 
$$P_i = S_i / (S_i \cap K_i) \cong (S_i / (S_i \cap K_j)) /( (S_i \cap K_i) / (S_i \cap K_j)).$$
This means that $S_i^{\Omega_j} \cong S_i/(S_i\cap K_j)$ is a normal subgroup of $G^{\Omega_j}$ which contains a chief factor isomorphic to $P_i \cong T^s \cong P_j$. Applying Lemma~\ref{l:ITgps} to $G^{\Omega_j}$, we see that $S_i^{\Omega_j}$ contains a plinth of $G^{\Omega_j}$. Since $G^{\Omega_j}$ has a unique plinth, namely $P_j$, it follows that there exists $M$ normal in $G$ such that $S_i \cap K_j \leqslant M \leqslant S_i$ and $M^{\Omega_j}= P_j$. By the minimality of $S_j$, we have $S_j \leqslant M \leqslant S_i$, and since $|S_i | = |S_j|$, we have $S_i = M = S_j$, and hence also $S_i^{\Omega_j} = S_j^{\Omega_j} = P_j$,   and the claim is proved.

\begin{claim} \label{claim:minnml}
    The groups $S_i$ are minimal normal subgroups of $G$.
\end{claim}

\medskip
Let $X$ be a minimal normal subgroup of $G$ with $X\leqslant S_i$. Since $X$ is nontrivial, there exists $j$ such that $X^{\Omega_j} \neq 1$. Then $1\ne X^{\Omega_j} \leqslant S_i^{\Omega_j}$, and it follows from \ref{clam:projs plinth} that $X^{\Omega_j}=S_i^{\Omega_j} = P_j$ and $S_i=S_j$.   By the minimality of $S_j$, we have $ S_j \leqslant X \leqslant S_i  = S_j$, and so $S_i=X$ and therefore $S_i$ is a minimal normal subgroup of $G$.

\medskip
An immediate consequence of \ref{claim:minnml} is that $S_i\cap K_i=1$ and  $S_i \cong S_i^{\Omega_i} = P_i \cong T^s$. Together with the uniqueness from \ref{clam:projs plinth}, this gives part (a) of the lemma. 
%Thus $S_i$ is characteristically simple for all $1 \leqslant i \leqslant n$.

We define a relation $\sim$ on the set $\{1,\ldots,n\}$ as follows: $i\sim j$ if and only if $S_i = S_j$. Clearly this is an equivalence relation, and the relation is preserved by $A$. Thus, for some positive integer $t$ that divides $n$, there are exactly $t$ equivalence classes of $\sim$. After relabelling, we may assume that $S_1$, \ldots, $S_t$ are representatives of the equivalence classes.   Since each $S_i$ is a  nonabelian minimal normal subgroup of $G$, we have 
$$
S:= \langle S_1,\ldots, S_n \rangle = \langle S_1,\ldots,S_t \rangle \cong S_1 \times \cdots \times S_t
$$
and since each $S_i \cong T^s$, the group $S \cong T^{st}$ is characteristically simple.  Let $1 \leqslant i \leqslant t$. For any $1\leqslant j \leqslant t$ with $j \neq i$, we have $S_i \neq S_j$, so \ref{clam:projs plinth} implies that $S_j^{\Omega_i}=1$. Hence $S^{\Omega_i} =  S_1^{\Omega_i} \times \cdots \times S_t^{\Omega_i} = S_i^{\Omega_i}$. This completes the proof.
\end{proof}

Armed with the structure provided by Lemma~\ref{l:soc}, we can now complete our analysis of the case where $G^{\Omega_i}$ has a unique nonabelian plinth. Theorem~\ref{t:GmustbeHSHC} below depends on the Classification of the Finite Simple Groups, since it depends on the Guralnick-Saxl Theorem (see Theorem~\ref{t:guralsaxl}).

\begin{theorem}\label{t:GmustbeHSHC}
Assume that Hypothesis~\ref{h:main2} holds. Suppose that each $G^{\Omega_i}$ has a unique plinth that is  nonabelian. Then $P_A(G)\neq P_A(U)$.    
\end{theorem}
\begin{proof}
% \begin{claim}
%     If $i \nsim j$, then for all $k$, either $S_i^{\Omega_k} = 1$ or $S_j^{\Omega_k}=1$.
% \end{claim}

% Since $i \nsim j$,  we have $S_i \neq S_j$ and so   \ref{claim:minnml} shows that $S_i$ and $S_j$ commute. By \ref{clam:projs plinth} we have  $S_i^{\Omega_k} = 1$ or $P_k$ and $S_j^{\Omega_k}=1$ or $P_k$. If  $S_i^{\Omega_k} = P_k = S_j^{\Omega_k}$, then $P_k=[P_k,P_k] = [S_i^{\Omega_k}, S_j^{\Omega_k}]=[S_i,S_j]^{\Omega_k} = 1$, a contradiction that proves the claim.
We will assume for a contradiction that $P_A(G)=P_A(U)$. Let $S=S_1\times \cdots \times S_t$ be the subgroup defined by Lemma~\ref{l:soc}. By that same lemma, $S_1$ is the unique minimal normal subgroup of $G$ contained in $S$ such that $S^{\Omega_1}=S_1^{\Omega_1}$. In particular, $S_j^{\Omega_1}=1$ for all $j$ such that $2\leqslant j\leqslant t$. Hence $S \cap K_1 = S_2 \times \cdots \times S_t$.

    By the definition of $\Omega_1$, the group $U$ contains $K_1$, and since we are assuming that $\Core_A(U)=1$, it follows that $U$ does not contain $S_1$. Hence $U\cap S=(U\cap S_1)\times K_1$ and $U\cap S_1$ is a proper subgroup of $S_1$.  
Also,  $A$ permutes transitively  by conjugation the minimal normal subgroups $S_i$ of $G$ that are contained in $S$. 
For $1\leqslant i\leqslant t$, choose an element $a_i\in A$ such that $S_1^{a_i}=S_i$, and take $a_1=1$. Let $A_1 = N_A(S_1)$.
%Let $U_i:=U^{a_i}$ so $U_1=U$, let $A_1:=N_A(S_1)$, and note that $N_A(S_i)=A_1^{a_i}$. 
The conjugation action of $A$ on $S$ induces a homomorphism $\varphi: A\to \Aut(S)$ with the property that each element of $A$ induces an automorphism which permutes the $S_i$ and hence lies in $\Aut(T^s)\wr \Sym(t)$. 
We identify each of the $S_i$ with $T^s$ so that we may write 
\begin{align}\label{eq:desc ai}
\varphi(a_i)=(\alpha_{i1},\dots,\alpha_{it})\sigma_i \text{ with each } \alpha_{ij}\in\Aut(T^s) \text{ and }\sigma_i\in \Sym(t) \text{ such that }1^{\sigma_i}=i. 
\end{align}

\begin{claim}\label{claim:pa1s1=pa1ucaps1}
    $P_{A_1}(S_1)=P_{A_1}(U\cap S_1)$.
\end{claim}  

Clearly $P_{A_1}(U \cap S_1) \subseteq P_{A_1}(S_1)$, so we just need to prove the reverse inclusion. Let $y_1 \in S_1$ be of prime power order. 
%Suppose to the contrary that this is not the case. Then there exists an element $y_1\in S_1$ of prime power order such that, for all $a\in A_1$, we have $y_1^a\not\in U\cap S_1$. 
Then, for $i=1,\dots, t$, the element $y_i:= y_1^{a_i}\in S_i$, and from the form of $\varphi(a_i)$ given in \eqref{eq:desc ai} we see that $y_i$ is the element $y_1^{\alpha_{i1}}$ of $T^s$. In particular $|y_i|=|y_1|$, so $y:=(y_1,y_2,\dots,y_t)\in S$ has prime power order $|y|=|y_1|$. 
Since  $P_A(G)=P_A(U)$, there exists $a\in A$ such that $y^a\in U$, and since $S$ is $A$-invariant, $y^a\in U\cap S$. As we saw above, $U\cap S$ contains $K_1=S_2\times\dots\times S_t$, and hence $y^a\in U\cap S$ if and only if its first entry  lies in $U\cap S_1$. Now $\varphi(a)=(\beta_1,\dots,\beta_t)\sigma$ for
some $\beta_i\in\Aut(T^s)$ and $\sigma\in \Sym(t)$. Suppose that $i^\sigma=1$, that is $S_i^a=S_1$. Then the first entry of $y^a$ is  
\[
y_{1\sigma^{-1}}^{\beta_{1\sigma^{-1}}} = y_i^{\beta_i},\ \mbox{which is the element 
$y_1^{\alpha_{i1}\beta_i}$ of $T^s$,}
\]
and hence $y_1^{\alpha_{i1}\beta_i}\in U\cap S_1$. Note that the element $a_ia\in A$ satisfies $S_1^{a_ia}=(S_1^{a_i})^a=S_i^a=S_1$, that is to say,  $a_ia\in A_1$, and
\[
\varphi(a_ia)=(\alpha_{i1},\dots,\alpha_{it})\sigma_i \cdot
(\beta_1,\dots,\beta_t)\sigma = 
(\alpha_{i1}\beta_{1\sigma_i},\dots,\alpha_{it}\beta_{t\sigma_i})\sigma_i\sigma.
\]
Since $1^{\sigma_i}=i$, this implies that $(y_1, 1,\dots,1)^{a_ia}=(y_1^{\alpha_{i1}\beta_i},1,\dots,1)$, and lies in $U \cap S$, with $a_ia\in A_1$. 
%hence $(y_1, 1,\dots,1)^{a_ia}\in U\cap S_1$ with $a_ia\in A_1$. 
%Removing the $t$-tuple notation this says that $y_1^{a_ia}\in U\cap S_1$, which is a contradiction. 
Removing the $t$-tuple notation this says that $y_1^{a_ia}\in U\cap S_1$. 
Thus $P_{A_1}(S_1)=P_{A_1}(U\cap S_1)$ as claimed.

\medskip

Now we write $S_1=T_1\times \dots \times T_s=T^s$. Since $S_1$ is transitive on $\Omega_1$ we have $G=US_1$ and since $S_1$ is a minimal normal subgroup of $G$, it follows that $U$ acts transitively on $\{T_1,\dots, T_s\}$ by conjugation. Since $U \cap S_1$ is a proper subgroup of $S_1$, after relabelling the $T_i$ if necessary, we may assume that $U\cap T_1$ is a proper subgroup of $T_1$.

\begin{claim}\label{claim:a11 on u cap t1}
    $P_{A_{11}}(T_1)=P_{A_{11}}(U\cap T_1)$, where $A_{11}=N_A(T_1)$.
\end{claim}  

By definition $P_{A_{11}}(U\cap T_1)\subseteq P_{A_{11}}(T_1)$, and to prove the reverse inclusion let $y\in T_1$ be an element of prime power order, and identify $y$ with the $s$-tuple  $(y,1,\dots,1)\in S_1$. Since $P_{A_1}(S_1)=P_{A_1}(U\cap S_1)$ by \ref{claim:pa1s1=pa1ucaps1}, there exists $a\in A_1$ such that $(y,1,\dots,1)^a\in U\cap S_1$. Now $T_1^a=T_i$ for some $i\leqslant s$, and there exists $u\in U$ such that $T_i^u=T_1$ since $G=US_1$. Thus $T_1^{au}=T_1$, that is $au\in A_{11}$, and therefore $(y,1,\dots,1)^{au}\in T_1$. Also  $(y,1,\dots,1)^{au}\in (U\cap S_1)^u=U\cap S_1$, so $(y,1,\dots,1)^{au}\in U\cap T_1$. Thus  $P_{A_{11}}(T_1)=P_{A_{11}}(U\cap T_1)$ and \ref{claim:a11 on u cap t1} is proved.

\medskip

 As noted above, $U\cap T_1$ is a proper subgroup of $T_1$, and $T_1$ is a finite nonabelian simple group. Let $\rho : A_{11}\rightarrow \aut(T_1)$ be the group homomorphism induced by the conjugation action on $T_1$.  Since $\rho$ restricted  to $T_1$ is a monomorphism, we have $$\rho(U \cap T_1) < \rho (T_1) \unlhd \rho(A_1) \leqslant \aut(T_1).$$ 
 The equality $P_{A_{11}}(T_1) = P_{A_{11}}(U \cap T_1)$ implies that $P_{\rho(A_{11})}(\rho(T_1)) = P_{\rho(A_{11})}(\rho(U \cap T_1))$    which 
is a contradiction to  the Guralnick--Saxl Theorem (Theorem~\ref{t:guralsaxl}).   This final contradiction completes the proof of Theorem~\ref{t:GmustbeHSHC}.
\end{proof}

\subsection{Two transitive minimal normal subgroups}

We  now assume that each $G^{\Omega_i}$ has two plinths, that is, two transitive minimal normal subgroups. This means that $G^{\Omega_i}$ is a primitive group (see Theorem~\ref{t:case div of IT}) and $\soc(G^{\Omega_i})\cong T^\ell$ for some finite nonabelian simple group $T$ and positive integer $\ell$. Our first step is to analyse the action of $\Soc(G)$ on $\Omega$
% of 
% \begin{equation}\label{e:S}
%     S:=\soc(G)=S_1 \times \cdots \times S_t,
% \end{equation}
and in particular the action of $\soc(G)$ on the different $G$-orbits $\Omega_1$, \ldots, $\Omega_n$.
%, where $S_1,\ldots,S_t$ are the minimal normal subgroups of $G$. 
%We will work with the set $\mathcal S = \{S_1,\ldots,S_t \}$. 
Recall that, under a group homomorphism,  the image of a minimal normal subgroup of $G$ is either trivial or a minimal normal subgroup of the image of $G$. In our context, each minimal normal subgroup of $G$
%$S_i$ 
acts either trivially or faithfully, on each ${\Omega_i}$. In particular, $\soc(G)^{\Omega_i} \leqslant \soc(G^{\Omega_i})$. Further, since $\soc(G)$ acts faithfully on $\Omega$, we have  $\soc(G) \lesssim \prod_{i=1}^n \soc(G^{\Omega_i}) \cong T^m$, for some finite nonabelian simple group $T$ and $m\geqslant 1$. In the next lemma, we sharpen this description.

\begin{lem}\label{l:2min}
     %For each $i$, the group $G^{\Omega_i}$ has exactly two plinths. Moreover, t
  Suppose that $G^{\Omega_i}$ is a primitive group with two minimal normal subgroups. Then   for each $1\leqslant i \leqslant n$, there are distinct minimal normal subgroups $L$ and $R$ of $G$ such that   $$\Soc(G^{\Omega_i}) = L^{\Omega_i} \times R^{\Omega_i}=\Soc(G)^{\Omega_i}$$ with $L\cong R \cong T^s$ for some finite nonabelian simple group $T$ and $s\geqslant 1$.
\end{lem}

\begin{proof}
    % We may suppose for a contradiction that $G^{\Omega_i}$ has a unique plinth. 
    % {\color{red}From Lemma~\ref{l:soc} we have that $S^{\Omega_1} = S_1^{\Omega_1}$ is the unique plinth of $G^{\Omega_1}$. 
%    Since $A$ is transitive on the $\Omega_i$, it suffices to prove the lemma with $i=1$.
Recall that $K_i$ is the kernel of the action of $G$ on $\Omega_i$. Thus each $K_i$ is normal in $G$ and $A$ acts transitively on the set $\{K_1,\ldots,K_n\}$. Since $G^{\Omega_1}$ contains exactly two minimal normal subgroups, by \cite[Corollary 3.11]{PS}, we may label these two minimal normal subgroups as $P_{1,L}$ and $P_{1,R}$ and we have  $P_{1,L}\cong P_{1,R} \cong T^s$, for some finite nonabelian simple group $T$ and positive integer $s$. Moreover, $P_{1,L}$ and $P_{1,R}$ are the left and right regular actions of $T^s$, respectively, and the \emph{double-centraliser} property holds: $C_{G^{\Omega_1}}(P_{1,L}) = P_{1,R}$ and $C_{G^{\Omega_1}}(P_{1,R}) = P_{1,L}$ (see \cite[Theorem 3.6(iii)]{PS}. We define $E_{1,L}$ and $E_{1,R}$ to be the full preimages in $G$ of $P_{1,L}$ and $P_{1,R}$ respectively. In particular, each of $E_{1,L}$ and $E_{1,R}$ is normal in $G$ and for any subgroup $H$ of $G$ with $H^{\Omega_i} = P_{1,L}$, we have $H \leqslant E_{1,L}$. Similarly for $E_{1,R}$.
Let 
\begin{equation}\label{eq:transversal}
    \mathcal T = \{a_1,\ldots,a_n\} 
\end{equation} 
be a transversal to $G$ in $A$ labelled such that $\Omega_1^{a_i} = \Omega_i$ for each $i$. In particular, we have $a_1 \in G$. For each $1\leqslant i \leqslant n$, we define
$$E_{i,L}:=(E_{1,L})^{a_i} \quad \text{and} \quad E_{i,R}:=(E_{1,R})^{a_i}.$$

\begin{claim}
The definitions of $E_{i,L}$ and $E_{i,R}$ do  not depend on the choice of transversal $\mathcal T$ and   $A$ acts transitively, by conjugation, on the sets $\{E_{1,L},\ldots,E_{n,L}\}$ and $\{E_{1,R},\ldots,E_{n,R}\}$.    
\end{claim}

Suppose that $\mathcal T' = \{a_1',\ldots,a_n'\}$ is another transversal to $G$ in $A$, again labelled so that $\Omega_1^{a_i'}=\Omega_i$. Since $\Omega_i=\Omega_1^{a_i}$ and $A$ acts regularly on $\{\Omega_1,\ldots,\Omega_n\}$ with stabiliser $G$, we have  $a_i' = g_i a_i$ for some $g_i\in G$. Hence $(E_{1,L})^{a_i'} = ((E_{1,L})^{g_i})^{a_i} = (E_{1,L})^{a_i} = E_{i,L}$, and similarly  $(E_{1,R})^{a_i'} =  E_{i,R}$. Thus the definition of $E_{i,L}$ and $E_{i,R}$ is indeed independent of the choice of transversal.
% 
%Then  $a_i$ induces a permutation equivalence between $G^{\Omega_1}$ and $G^{\Omega_i}$ via $g^{\Omega_1} \mapsto (g^{a_i})^{\Omega_i}$. In particular, if $H \leqslant G$ is such that $H^{\Omega_1}$ is transitive, then $H^{a_i}$ is transitive on $\Omega_i$. 
%
%Now for each $i \leqslant n$, we define $P_{i,L}$ and $P_{i,R}$ to be the images of $P_{1,L}$ and $P_{1,R}$ under the permutation equivalence induced by the element $a_i$ defined in \eqref{eq:transversal}. We claim that this definition does not depend on the choice of transversal. For any element $a\in A$ such that $\Omega_1^a = \Omega_i$, we may write $a=ga_i$ for some $g\in G$. Now for any  $h\in G$ such that $h^{\Omega_1} \in P_{1,L}$ we have that $h^a = h^{ga_i}=(h^g) ^{a_i} = (h')^{a_i}$ for some $h'\in G$ such that $(h')^{\Omega_1}\in P_{1,L}$, and so $(h^{\Omega_1})^a = (h^a)^{\Omega_i} = ((h')^{a_i})^{\Omega_i} \in P_{i,L}$. 
To prove the second assertion, let $a\in A$ and $i$ be arbitrary and consider $(E_{i,L})^a$. Let $j$ be such that $\Omega_i^a = \Omega_j$. Now $\Omega_i = \Omega_1^{a_i}$ and $\Omega_j=\Omega_1^{a_j}$. Thus $\Omega_1^{a_i a a_j^{-1} } = \Omega_1$. Since $A$ acts regularly on $\{\Omega_1,\ldots,\Omega_n\}$ with stabiliser $G$, there exists $g\in G$ such that $a_i a a_j^{-1} = g$, whence $a = a_i^{-1} g a_j$. Then, since $E_{1,L}$ is normal in $G$, 
$$
(E_{i,L})^a = (E_{i,L})^{a_i^{-1}g a_j} = ((E_{1,L})^g)^{a_j} = (E_{1,L})^{a_j} = E_{j,L}.
$$
It follows that $A$ acts transitively on the set $\{E_{1,L},\ldots,E_{n,L}\}$ by conjugation. Exactly the same argument shows that $A$ acts transitively on the set $\{E_{1,R},\ldots,E_{n,R}\}$ by conjugation.
%, and we note that these actions are equivalent to the $A$-action on $\{\Omega_1,\ldots,\Omega_n\}$.

\medskip

The previous claim allows us to define 
$$\mathcal L _i : = \{ X \unlhd G \mid X^{\Omega_i} = (E_{i,L})^{\Omega_i} = P_{i,L} \} \quad \text{and}\quad \mathcal R_i := \{ X \unlhd G \mid X^{\Omega_i} = (E_{i,R})^{\Omega_i} = P_{i,R} \}$$
and we note that each of the sets $\{\LL_1,\ldots,\LL_n\}$ and $\{\R_1,\ldots,\R_n\}$ is permuted transitively by $A$. Furthermore, $E_{i,L} \in \mathcal L_i$ and $E_{i,R} \in \mathcal R_i$, so that $\LL_i , \R_i$ are both non-empty.

\begin{claim}
    For each $i$, the sets $\mathcal L_i$ and $\R_i$ are closed under intersection.
\end{claim}

The proof is the same as the proof of  \ref{claim:closedunderintersection}  in the proof of Lemma~\ref{l:soc} and so we omit it here.

\medskip

For each $i$ we set $L_i := \bigcap_{X \in \LL_i} X$ and $R_i := \bigcap_{X\in \R_i} X$, and note that $A$ permutes each of the  sets $\{L_1,\ldots,L_n\}$ and $\{R_1,\ldots,R_n\}$ transitively. Hence $|L_i|=|L_j|$ for all $i,j$ and $|R_i|=|R_j|$ for all $i,j$. Furthermore, since $L_i\in\LL_i$ and $R_i\in\R_i$, we have  $L_i^{\Omega_i}=P_{i,L}$ and $R_i^{\Omega_i}=P_{i,R}$ so, in particular, $L_i \neq  R_i$, and both $L_i$ and $R_i$ are nontrivial, for all $i$.

\begin{claim}\label{claim:L_i min nml}
After possibly swapping the labels $L$ and $R$, we may assume that, for all $i$, $L_i$ is a minimal normal subgroup of $G$.
\end{claim}

Let $X$ be a  minimal normal subgroup of $G$. Since $G$ acts faithfully on $\Omega$, $X$ must act nontrivially on some $\Omega_i$, that is, $X^{\Omega_i}\neq 1$ and hence $X^{\Omega_i}$ is a minimal normal of $G^{\Omega_i}$. Thus $X^{\Omega_i} = P_{i,L}$ or $P_{i,R}$. In the first case, we have  $X \in \LL_i$, and so $L_i \leqslant X$, which implies $L_i=X$. In this case, for each $j\leq n$, there is an element $a\in A$ such that $\Omega_i^a=\Omega_j$, and the conjugate $X^a$ is a minimal normal subgroup of $G$ lying in $\LL_i^a=\LL_j$, which implies that $L_j\leq X^a$ and hence $L_j=X^a$.   In the second case, a similar argument yields that, for each $j$, $R_j$ is some $A$-conjugate of $X$. This proves the claim, after interchanging the labels $L$ and $R$ in the second case.

\begin{claim}\label{claim:proj to min nml}
    For all $i,j$, one of  the following holds:
    \begin{center}
        (a) $R_i^{\Omega_j}=1$,\quad or (b) $R_i^{\Omega_j}=P_{j,R}$ and $R_i=R_j$,\quad or (c) $R_i^{\Omega_j}=P_{j,L}$.
    \end{center}
\end{claim}

To prove the claim, we may assume $R_i^{\Omega_j} \neq 1$ (i.e.~part (a) does not hold). Thus $R_i^{\Omega_j}$ contains one of the two minimal normal subgroups of $G^{\Omega_j}$.  Suppose first that $R_i^{\Omega_j} \geqslant P_{j,R}$. By the Correspondence Theorem, we may choose $R_0 \leqslant R_i$ with $R_0$ normal in $G$ such that $R_0^{\Omega_j}=P_{j,R}$. Then $R_0 \in \mathcal R_j$ and so by the minimality of $R_j$, we have $R_j \leqslant R_0 \leqslant R_i$. The equality $|R_i|=|R_j|$ implies that $R_i=R_j$ and  $R_i^{\Omega_j}=R_j^{\Omega_j} = P_{j,R}$, as in part (b). Now  we may assume that  $R_i^{\Omega_j} \ngeqslant P_{j,R}$. The fact that $P_{j,R}$ is a minimal normal subgroup of $G^{\Omega_j}$ combined with the fact that $R_i^{\Omega_j}$ is normal in $G^{\Omega_j}$, implies that  $R_i^{\Omega_j}\cap P_{j,R} = 1$, and hence $R_i^{\Omega_j} \leqslant C_{G^{\Omega_j}}(P_{j,R})=P_{j,L}$. This implies that $R_i^{\Omega_j} = P_{j,L}$, as in part (c).

% \begin{claim}
%     For any $j$ such that  $R_i^{\Omega_j} \ngeqslant P_{j,L}$, we have that $R_i^{\Omega_j}=1$ or $R_i^{\Omega_j}=P_{j,R}$.
% \end{claim}

\begin{claim}\label{claim:proj to R}
    Fix $i\leq n$. If, for all $j$,~\ref{claim:proj to min nml}(a) or (b) holds, that is, for all $j$ we have that $R_i^{\Omega_j} =1$ or $P_{j,R}$, then $R_i$ is a minimal normal subgroup of $G$.
\end{claim}

Let $X$ be a minimal normal subgroup of $G$ contained in $R_i$. Now $X$ must act nontrivially on $\Omega_j$ for some $j$, so we have $1\neq X^{\Omega_j} \leqslant R_i^{\Omega_j}$. Then by the assumptions in this claim, $X^{\Omega_j}=R_i^{\Omega_j}=P_{j,R}$. The minimality of $R_j$  means that $R_j \leqslant X \leqslant R_i$, and since $|R_j|=|R_i|$, we have that $R_i=X=R_j$. In particular, $R_i$ is a minimal normal subgroup of $G$.

\medskip

\begin{claim}\label{claim:cover L}
    %If there are $i$, $j$ such that   $R_i^{\Omega_j} = P_{j,L}$, then $R_i$ is a minimal normal subgroup of $G$.
    Fix $i\leq n$. If~\ref{claim:proj to min nml}(c) holds  for some $j$, that is, $R_i^{\Omega_j}=P_{j,L}$, then $R_i$ is a minimal normal subgroup of $G$.
\end{claim}

Suppose that $R_i^{\Omega_j}=P_{j,L}$.
%and let $R_{i,0} \leqslant R_i$ be the normal subgroup of $G$ such that $R_{i,0}^{\Omega_j}=P_{j,L}$. 
By the minimality of $L_j$ we have $L_j \leqslant R_i$ and since $L_j$ is a minimal normal subgroup of $G$ (by \ref{claim:L_i min nml}) and $L_j$  acts faithfully on $\Omega_j$, we obtain
$$R_i = L_j \times (R_i \cap K_j).$$
Now consider $L_j^{\Omega_i}$. Since $L_j$ is normal in $G$, and $L_j \leqslant R_i$, we have $L_j^{\Omega_i} = 1$ or $L_j^{\Omega_i}=R_i^{\Omega_i}=P_{i,R}$. In the latter case, this means that $L_j \in \R_i$, and so by the minimality of $R_i$, we obtain $R_i \leqslant L_j$, whence $R_i=L_j$, as required. So we may suppose that $L_j^{\Omega_i} = 1$. This implies that 
$$
(R_i  \cap K_j)^{\Omega_i}=R_i^{\Omega_i}=P_{i,R}
$$
and by minimality of $R_i$, we have $R_i = R_i \cap K_j$, which means that $L_j=1$, a contradiction.

\medskip

By~\ref{claim:proj to min nml}, for each $i$, the assumptions of one of  \ref{claim:proj to R} or \ref{claim:cover L} holds. We therefore conclude that $R_i$ 
%(as well as $L_i$)
is a minimal normal subgroup of $G$ for each $i \leqslant n$.
Since ~\ref{claim:L_i min nml} shows that each $L_i$ is a minimal normal subgroup of $G$, and $R_i\neq L_i$, we have $L_i \times R_i \leqslant \soc(G)$.
Moreover $R_i^{\Omega_i} = P_{i,R} \neq P_{i,L} = L_i^{\Omega_i}$ so $\soc(G^{\Omega_i}) = L_i^{\Omega_i} \times R_i^{\Omega_i} = (L_i \times R_i)^{\Omega_i} \leqslant  \Soc(G)^{\Omega_i}.$ On the other hand, $\soc(G)^{\Omega_i} \leqslant \Soc(G^{\Omega_i})$, and so $\soc(G^{\Omega_i}) = \Soc(G)^{\Omega_i}$, and the proof is complete.
\end{proof}

It follows from Lemma~\ref{l:2min} that each minimal normal subgroup of $G$ is isomorphic to $T^s$, for some finite nonabelian simple group $T$ and fixed positive integer $s$. Thus, if $G$ has exactly $t$ minimal normal subgroups $S_1, \ldots, S_t$, then 
$$S:=\Soc(G) = S_1 \times \cdots \times S_t  \cong T^m$$
with $m=st$. We may assume that
%, we may assume that 
$S^{\Omega_1}=S_1\times S_2$.
%, with each $S_i\cong T^s$ so that
%, by \eqref{e:S},
%$S\cong T^m$ with $m=st$.
Further, by \cite[Theorem 3.6]{PS}, each of $S_1$ and $S_2$ acts regularly on $\Omega_1$, and hence
\begin{equation}\label{e:index}
    |G:U|=|\Omega_1|=|S_1|=|S_2|=|T|^s.
\end{equation}
Recall that our task is to find an upper bound for $|G:U|$ that is a function of $n=|A:G|$ when $P_A(G)=P_A(U)$. Thus we need to bound both $s$ and $|T|$ by functions of $n$. To do this we exploit the new upper bound on $|T|$ in terms of $m(T)$ (see \eqref{mT}) given  by Theorem~\ref{t:mT}. 
The following crucial lemma shows how to bound $s$ and $m(T)$ in terms of $n$. The ideas in the proof come from \cite[Proposition 3.1]{CP91}, which originated in \cite[Proposition 2.2]{CP89}. They involve,  among other things, the notion of vertex-colourings of graphs discussed in Section~\ref{s:prelim}. First we set out some notation.

 Recall that $U=A_\alpha=G_\alpha$, the stabiliser  of the point $\alpha=U\in\Omega$, and by \cite[Section 3.3]{PS}, $(U\cap S)^{\Omega_1}$ is a diagonal subgroup of $S^{\Omega_1} = S_1\times S_2$, and we may assume that 
%$(U\cap S)^{\Omega_1} = D = \{(x,x)\mid x\in T^s\} < S_1\times S_2$, the `straight diagonal' subgroup of $S_1\times S_2$.  
\begin{equation}\label{e:D}
    U\cap S = D\times \prod_{2 < i \leqslant t} S_i,\quad \text{where}\quad D = \{(x,x)\mid x\in T^s\} < S_1\times S_2, 
\end{equation}
that is, $D$ is the `straight diagonal' subgroup of $S_1\times S_2$. 
%Also $t\geq2$ by Lemma~\ref{l:Sunf}.

%{\color{red}DELETED $t\geqslant 3$ here, assuming we can use $t \geqslant 2$!}

\begin{lemma}\label{l:2min2} 
Suppose that $P_A(G)=P_A(U)$.
     Then, with  $m(T)$ as in \eqref{mT}, 
     $$\max\{ s, m(T)-1\}\leqslant s\cdot (m(T)-1)\leqslant n.$$
\end{lemma}

\begin{proof}
    Let $p$ be a prime divisor of $|T|$ such that the number of $\Aut(T)$-classes of $p$-elements in $T$ is equal to $m(T)$, and recall that the identity $1_T$ is a $p$-element. Thus there are at least two distinct $\Aut(T)$-classes of $p$-elements in $T$, so $m(T)\geqslant 2$. Hence 
    \begin{equation} \label{eq:max lt product} 
    \max\{ s, m(T)-1\} \leqslant s\cdot (m(T)-1), 
    \end{equation}
    so to establish the lemma, we need to bound the right hand side of the above inequality by~$n$. Recall that $S = S_1 \times \cdots \times S_t$, and each $S_i \cong T^s$. 
    
   As usual we write elements of $S$ as $t$-tuples $x=(x_1,\dots, x_t)$  with $x_i\in S_i$, for each $i$.   Let $\mathcal{T}$ be the set of all $\Aut(T^s)$-conjugacy classes of $p$-elements in $T^s$.
    
    \medskip\noindent
    \emph{Claim 1: $s\cdot (m(T)-1) \leqslant |\mathcal{T}|$.}
    
   \smallskip\noindent
    This is a very weak lower bound but sufficient for our purposes. To see that it holds, choose representatives $y_1,\dots, y_{m(T)-1}$ for the $m(T)$ distinct $\Aut(T)$-classes of nontrivial $p$-elements in $T$, and for  $1\leqslant \ell\leqslant m(T)-1$ and $1\leqslant j\leqslant s$, let $x_{\ell,j}\in T^s$ be the element with entries $1,\dots,j$ equal to $y_\ell$ and entries $j+1,\dots, s$ equal to the identity. Then the $x_{\ell,j}$ lie in pairwise distinct $\Aut(T^s)$-classes. This proves Claim 1.

    \medskip
We define   $\mathcal{S} = \{S_1,\ldots,S_t\}$.   By Lemma~\ref{l:2min} we may assume that $S^{\Omega_1}= S_1^{\Omega_1} \times S_2^{\Omega_1} \cong  S_1 \times S_2$.
    %, and by  Lemma~\ref{l:Sunf}, the group $S$ acts unfaithfully on $\Omega_1$.  Thus, for some $\ell\leqslant t$, the direct factor $S_\ell$ of $S$ satisfies $S_\ell^{\Omega_1} =1 $, and so $\ell\notin \{1,2\}$.
    In particular,  $t=|\mathcal{S}|\geqslant 2$.
Each $p$-element $x=(x_1,\dots, x_t)\in S$ determines a map 
\begin{equation} \label{eq:defn of phi_x}
\phi_x:\mathcal{S}\to\mathcal{T},\ \mbox{by defining $\phi_x(S_i)$ as the $\Aut(T^s)$-class containing $x_i \in S_i$.}
\end{equation}
As $x$ runs over the $p$-elements in $S$ we clearly obtain all maps $\mathcal{S}\to\mathcal{T}$. 

 \medskip\noindent
    \emph{Claim 2: for each $p$-element $x=(x_1,\dots, x_t)\in S$, there exist $i,j\leqslant t$ and $k\leqslant n$ such that $i\ne j$,  $\phi_x(S_i)=\phi_x(S_j)$, and $S^{\Omega_k} = S_i^{\Omega_k} \times S_j^{\Omega_k} \cong S_i \times S_j$.}\quad  

   \smallskip\noindent
    From Lemma~\ref{l:P_A(S) = P_A(UcapS)} we have   $P_A(S)=P_A(U \cap S)$, so there exists $a\in A$ such that $x^a\in U\cap S$. Recall that each element of $A$, acting by conjugation, induces an automorphism of $S$, and this automorphism must permute the minimal normal subgroups $S_i$ of $G$ and hence must lie in the subgroup  $\Aut(T^s)\wr \Sym(t)$ of $\Aut(S)$ preserving the decomposition $S=S_1\times\dots\times S_t$. Suppose that the element $a$ induces the automorphism $(\alpha_1,\dots,\alpha_t)\sigma$ in $\Aut(T^s)\wr \Sym(t)$, where the $\alpha_i\in\Aut(T^s)$ and $\sigma\in \Sym(t)$. If $i, j$ are such that $S_i^\sigma=S_1$ and $S_j^\sigma = S_2$, then $i\ne j$, the first entry of $x^a$ is $x_i^{\alpha_i}$ and the second entry of $x^a$ is $x_j^{\alpha_j}$. Since $U \cap S = D\times \prod_{i=3}^t S_i$, as in \eqref{e:D}, it follows that $x_i$ and $x_j$ lie in the same $\Aut(T^s)$-class, that is, $\phi_x(S_i)=\phi_x(S_j)$.    Finally, for $k$ such that $(\Omega_k)^a = \Omega_1$, we have  $S^{\Omega_k} = (S^{\Omega_1})^{a^{-1}} = (S_1^{\Omega_1} \times S_2^{\Omega_1})^{a^{-1} } = S_i^{\Omega_k} \times S_j^{\Omega_k} \cong S_i \times S_j$. Thus Claim 2 is proved. %Just before Claim 1, we identified a direct factor $S_\ell$ such that $S_\ell^{\Omega_1}=1$. Hence, setting $k:=\ell^{\sigma^{-1}}$, we have}

    \medskip
\noindent \emph{Definition of the graph $\Gamma$:}    We define a graph $\Gamma$ with vertex set $\mathcal{S}$ such that $\{S_i, S_j\}$ is an edge of $\Gamma$ if and only if there exists $k\leqslant n$ such that $S^{\Omega_k}= (S_i\times S_j)^{\Omega_k}$. (Note that $k$ is not necessarily uniquely determined by $S_i, S_j$.)

    Thus $\Gamma$ has $t$ vertices and the number $E$ of edges satisfies  $E\leqslant n$.   
    Also $A$ acts as a group of automorphisms of $\Gamma$, and $A$ is transitive on the edge set since $A$ permutes the $\Omega_k$ transitively. Since each $S_i$ must act nontrivially on at least one $\Omega_k$, edge-transitivity implies that $A$ has at most two orbits on the $S_i$ and hence at most two orbits on vertices. Vertices in the same $A$-orbit are incident with the same number of edges. Thus either (i) $A$ is vertex-transitive and each vertex is incident with exactly $d$ edges, for some constant $d$, or (ii) $A$ has two orbits $\mathcal{S}_1, \mathcal{S}_2$ on vertices with $|\mathcal{S}_i|=t_i$ and $t=t_1+t_2$, and each vertex in $\mathcal{S}_i$ is incident with $d_i$ edges, for $i=1,2$ and constants $d_1, d_2$. In case (i), counting incident vertex-edge pairs gives $2E=dt$ with $t\geq 2$, % t at least 3 ???
    so 
    \begin{equation}\label{eq:bnd on k in i} 
    d = 2E/t\leqslant 2n/t\leqslant n.
    \end{equation}
    In case (ii), $\Gamma$ is bipartite, each edge is incident with a vertex from each set $\mathcal{S}_i$, and a similar count gives $E=d_1t_1=d_2t_2$, so each \( d_i  = E/t_i \leqslant n/t_i \leqslant n\) and hence 
    \begin{equation}\label{eq:bnd on k in ii} 
   % d_i= E/t_i\leqslant n/t_i\quad \text{and} \quad  \max\{t_1,t_2\}\geqslant 2 
    \max\{d_1,d_2\}\leqslant n. 
    \end{equation}

 \medskip\noindent
    \emph{Claim 3: in case (i), $|\mathcal{T}|\leqslant d$, and in case (ii), $|\mathcal{T}|\leqslant \max\{d_1,d_2\}$. }\quad 

    \smallskip\noindent
    %In case (i), a so-called greedy algorithm shows that a vertex-colouring exists for any set $\mathcal{X}$ of size at least $k+1$; namely, arbitrarily order the vertices and colour them sequentially so that adjacent vertices have different colours. When assigning a colour to a vertex $S_i$, at most $k$ previous vertices in the list will be adjacent to $S_i$. Hence at most $k$ colours cannot be used and there is always a colour available for $S_i$ since there are at least $k+1$ colours in $\mathcal{X}$. 
    For any map $\phi:\mathcal{S}\to\mathcal{T}$, there is an element $x\in S$ such that $\phi=\phi_x$ (see \eqref{eq:defn of phi_x}). By Claim 2, there exist $i \neq j$ such that $\phi(S_i) =\phi_x(S_i)= \phi_x(S_j) =\phi(S_j)$, and an integer $k$ such that $S^{\Omega_k}=S_i^{\Omega_k} \times S_j^{\Omega_k}$, so that $\{S_i,S_j\}$ is an edge of $\Gamma$. Hence for any assignment of the elements of $\mathcal{T}$ to `colour' the vertices of $\Gamma$, at least one edge of $\Gamma$ receives the same `colour' for both of its incident vertices, so  there is no proper vertex colouring of $\Gamma$ by $\mathcal T$.   On the other hand, Theorem~\ref{thm:colouring} shows that $\Gamma$ can be coloured with $d+1$ colours in case (i) and with $\max \{d_1,d_2\}+1$ colours in case (ii). In case (i) this means that   $|\mathcal{T}| <  d+1$, which gives $|\mathcal T | \leqslant d$, and in case (ii) we deduce that $|\mathcal{T}| < \max\{d_1,d_2\}+1$, that is, $|\mathcal T| \leqslant \max\{d_1,d_2\}$. This proves Claim 3. 
 
    \medskip
 In case (i), using  Claims 1 and 3 and \eqref{eq:bnd on k in i}, we obtain
%$$n> 2n/3\geqslant 2n/t \geqslant k\geqslant |\mathcal T | \geqslant s\cdot (m(T)-1),$$
$$s\cdot(m(T)-1) \leqslant |\mathcal T | \leqslant d% \leqslant 2n/t 
\leqslant  n  $$ % 2n/3 < n$$
as required.
In case (ii),  using  Claims 1 and 3 and \eqref{eq:bnd on k in ii},
%we obtainfrom \eqref{eq:bnd on k in ii} we have $n\geqslant d_it_i\geqslant d_i$ for each $i$, so $n\geqslant \max\{d_1,d_2\}$. Thus, from  Claims 1 and 3 we have
$$s\cdot (m(T)-1) \leqslant |\mathcal{T}| \leqslant \max\{d_1,d_2\} \leqslant n .$$
 This completes the proof of Lemma~\ref{l:2min2}.
\end{proof}

Finally we draw together the results of this section to prove Theorem~\ref{t:mainIT}.

\medskip\noindent\emph{Proof of Theorem~\ref{t:mainIT}.}\quad Let $A$ be a finite group with a normal subgroup $G$ of index $n$, and let $U$ be a proper subgroup of $G$ such that $P_A(G)=P_A(U)$ and  the action of $G$ on $[G:U]$, the set of cosets of $U$, is innately transitive.
%There is nothing to prove if $U=G$ so we assume that $U<G$.
By Lemma~\ref{l:cortriv}, we may assume that $\Core_A(U)=1$ and hence that $A$ acts faithfully and transitively on $\Omega=[A:U]= \{Ua\mid a\in A\}$. The group $G$ has $n$ orbits $\Omega_1,\dots,\Omega_n$  on $\Omega$ with $U\in\Omega_1$, and the $G$-action on $\Omega_1=[G:U]$ is innately transitive. If $G^{\Omega_1}$ is an affine group, then Lemma~\ref{l:affine} shows that $|G:U|\leqslant n$.  By Theorem~\ref{t:GmustbeHSHC}, we may assume that $G^{\Omega_1}$ has at least two plinths, and so $G^{\Omega_1}$ is a primitive group. Now $S=\Soc(G)=T^m$ for some nonabelian finite simple group $T$ and $m\geqslant 1$.
%, and if $T$ is cyclic then $|G:U|\leqslant n$. So we may assume that $T$ is a nonabelian simple group. 
By Lemma~\ref{l:2min}, $\Soc(G^{\Omega_1})=S_1^{\Omega_1}\times S_2^{\Omega_1}\cong S_1\times S_2$ for distinct minimal normal subgroups $S_1, S_2$ of $G$, and by \eqref{e:index} we have  $|G:U|= |S_1|=|S_2|=|T|^s$ for some $s\geqslant 1$. By Lemma~\ref{l:2min2}, $\max\{s, m(T)-1\}\leqslant n$, and so, for the increasing integer function $h$ of Theorem~\ref{t:mT}, we have
\[
|G:U|=|T|^s \leqslant h(m(T))^n\leqslant h(n+1)^n.
\]
This completes the proof of Theorem~\ref{t:mainIT}. \qed

\medskip

Recall that Theorem~\ref{t:main} follows immediately from Theorem~\ref{t:mainIT}. The assertion about the increasing integer function $f$ mentioned in Remark~\ref{r:main} is justified by the last displayed inequality in the proof above.


\begin{thebibliography}{abc} 

\bibitem{BP}
J. Bamberg and C.~E. Praeger, Finite permutation groups with a transitive minimal normal subgroup, Proc. London Math. Soc. (3) {\bf 89} (2004), no.~1, 71--103.

\bibitem{Br}
R. Brandl, A covering property of finite groups
\emph{Bull. Austral. Math. Soc.} {\bf23} (1981), 227--235.

\bibitem{BSW}
D. Bubboloni, P. Spiga and T.~S. Weigel, {\it Normal 2-Coverings of the Finite Simple Groups and their Generalizations}, Lecture Notes in Mathematics, 2352, Springer, Cham, 2024.

\bibitem{FKS}
B. Fein, W. M. Kantor and M. Schacher,
Relative Brauer groups. II., \emph{J. Reine Angew. Math.} {\bf328} (1981), 39--57.

\bibitem{FHS}
M. Fusari, S. Harper and  P. Spiga,
Kronecker classes, normal coverings and chief factors of groups, arXiv:2410.02569.


\bibitem{FPS}
%Marco Fusari, Andrea Previtali and Pablo Spiga, Cliques in derangement graphs for innately transitive groups, J. Group Theory, \textbf{27} (2024) 929--956. 
M. Fusari, A. Previtali and P. Spiga, Cliques in derangement graphs for innately transitive groups, J. Group Theory {\bf 27} (2024), no.~5, 929--965.


\bibitem{GMP24}
M. Giudici, L. Morgan, and C.~E.~Praeger, Simple groups have many classes of $p$-elements, preprint, 2024.


\bibitem{HoGT}
\emph{Handbook of graph theory}.
Second edition. Edited by Jonathan L. Gross, Jay Yellen and Ping Zhang
Discrete Math. Appl. (Boca Raton) CRC Press, Boca Raton, FL, 2014

\bibitem{G90}
R. M. Guralnick,
Zeros of permutation characters with applications to prime splitting and Brauer groups, \emph{J. Algebra} {\bf131} (1990), 294--302.

\bibitem{Isaacs}
I.~M. Isaacs, {\it Algebra: a graduate course}, reprint of the 1994 original, 
Graduate Studies in Mathematics, 100, Amer. Math. Soc., Providence, RI, 2009.

\bibitem{J77}
W. Jehne, Kronecker classes of algebraic number fields, \emph{J. Number Theory} {\bf 9} (1977), 279--320.

\bibitem{Jordan}
C.~Jordan, Recherches sur les substitutions, \emph{J. Liouville} 17 (1872), 351--367.

\bibitem{Kourovka}
`Unsolved Problems in Group Theory: The Kourovka Notebook', No. 19, (Mathematics Institute of the Siberian Division of the Academy of
Sciences of the USSR, Novosibirsk, 1990).

\bibitem{MRS}
K. Meagher, A. S. Razafimahatratra and P. Spiga, On triangles in derangement graphs, \emph{Journal of Combinatorial Theory, Series A } \textbf{180} (2021) 105390.

\bibitem{NZM}
I. Niven, H. S. Zuckerman, and H. L. Montgomery, 
\emph{An introduction to the theory of numbers}. 
John Wiley \& Sons, Inc., New York, 1991.

\bibitem{CP88}
C. E.~Praeger, Covering subgroups of groups and Kronecker classes of fields, \emph{J. Algebra} 118 (1988), 455--463.

\bibitem{CP89}
C. E.~Praeger, On octic field extensions and a problem in group theory. In \emph{Group Theory: Proceedings of the 1987 Singapore Conference} Walter de Gruyter, Berlin, New York, 1989. pp. 443-- 463.



\bibitem{CP91}
C. E.~Praeger, Kronecker classes of field extensions of small degree,  Journal of the Australian Mathematical Society 50 (1991), 297-- 315.

\bibitem{kron}
C. E.~Praeger, Kronecker classes of fields and covering
subgroups of finite groups,  Journal of the Australian Mathematical Society 57 (1994), 17-- 34.

\bibitem{PS}
C. E.~Praeger and Csaba Schneider, \emph{Permutation groups and cartesian decompositions},  London Math. Soc. Lecture Note Series, Vol. 449, Cambridge University Press, Cambridge, 2018. 

\bibitem{Saxl}
J.~Saxl, On a question of W. Jehne  concerning covering subgroups of groups and Kronecker classes of fields.
\emph{J. London Math. Soc. (2)} {\bf38} (1988), 243--249.

\bibitem{Serre}
J.-P. Serre. On a theorem of Jordan. \emph{Bull. Amer. Math. Soc.}, 40 (2003), 429--440.

    
\end{thebibliography}
\end{document}